\documentclass[11pt,a4paper]{article}
\usepackage{lipsum}
\usepackage{authblk}
\usepackage{amsmath,a4}  
\usepackage[utf8]{inputenc}
\usepackage{graphicx,enumerate,color,amsfonts,amsthm,amsmath,amssymb,mathrsfs,marginnote,dsfont,textcomp,tikz,marginnote,subfigure,pdflscape,multirow,booktabs,float,enumitem,scrtime,currfile,titleps,hyphenat,algorithm,algpseudocode,bigints,diagbox,multirow,mathtools,mdframed}
\usepackage[colorlinks,
    linkcolor={red!60!black},
    citecolor={black},%blue!80!black
    urlcolor={blue!80!black}]{hyperref}
\usepackage{tikz,pgf,pgffor,pgfplots}
\tikzset{>=latex}
\usetikzlibrary{calc,patterns,decorations,intersections,decorations.pathmorphing,fit,backgrounds,arrows,shapes,snakes,automata,petri,positioning,,graphs,graphs.standard}
\usepgfmodule{shapes,plot}

\newtheorem{theorem}{Theorem}[section]

\newtheorem{lemma}[theorem]{Lemma}

\theoremstyle{definition}

\theoremstyle{remark}
\newtheorem*{remark}{Remark}

\makeatletter
\newcommand{\setword}[2]{%
  \phantomsection
  #1\def\@currentlabel{\unexpanded{#1}}\label{#2}%
}
\makeatother

\newcommand{\dd}{\textnormal{d}}

\newcommand{\p}{\textnormal{\textbf{p}}}

\newcommand{\diag}{\textnormal{\textbf{diag}}}

\newcommand{\aev}{\textnormal{a.e.}}

\newcommand{\R}{\mathbb{R}}

\newcommand{\inner}[1]{\left\langle #1\right\rangle}
\newcommand{\norm}[1]{\left\lVert #1\right\rVert}
\newcommand{\abs}[1]{\left\lvert #1\right\rvert}
\newcommand{\bra}[1]{\left(#1\right)}
\newcommand{\set}[1]{\left\{#1\right\}}
\newcommand{\intt}{\int \!\!\!\!\! \int}
\newcommand{\LM}{\mathcal{L}}
\newcommand{\PP}{\textnormal{\textbf{P}}}
\newcommand{\PM}{\mathcal{P}}
\newcommand{\VM}{\mathcal{V}}

\allowdisplaybreaks
\usepackage[top=2cm, bottom=2cm, left=2cm, right=2cm]{geometry}
\usepackage{fancyhdr}
\newpagestyle{main}{
\setheadrule{.4pt}% Header rule
\sethead{}% left
        {}%                 center
        {\sectiontitle}%    right
}
\renewenvironment{abstract}{%
\hfill\begin{minipage}{0.95\textwidth}
\rule{\textwidth}{1pt}}
{\par\noindent\rule{\textwidth}{1pt}\end{minipage}}
\makeatletter
\renewcommand\@maketitle{%
\hfill
\begin{minipage}{0.95\textwidth}
\vskip 2em
\let\footnote\thanks 
{\Large \bf \@title \par }
\vskip 1.5em
{\large \@author \par}
\end{minipage}
\vskip 1em \par
}
\makeatother
\begin{document}
%
%title and author details
\title{\nohyphens{Estimation of time-space-varying parameters in dengue epidemic models}}
\author[a,$\ast$]{Karunia Putra Wijaya}
\author[a]{Thomas G\"{o}tz}
\affil[a]{\small\emph{Mathematical Institute, University of Koblenz, 56070 Koblenz, Germany}}
\affil[$\ast$]{Corresponding author. Email: \href{mailto:karuniaputra@uni-koblenz.de}{karuniaputra@uni-koblenz.de}}
\maketitle
\begin{abstract}
{\textbf{Abstract}}: There are nowadays a huge load of publications about dengue epidemic models, which mostly employ deterministic differential equations. The analytical properties of deterministic models are always of particular interest by many experts, but their validity -- if they can indeed track some empirical data -- is an increasing demand by many practitioners. In this view, the data can tell to which figure the solutions yielded from the models should be; they drift all the involving parameters towards the most appropriate values. By prior understanding of the population dynamics, some parameters with inherently constant values can be estimated forthwith; some others can sensibly be guessed. However, solutions from such models using sets of constant parameters most likely exhibit, if not smoothness, at least noise-free behavior; whereas the data appear very random in nature. Therefore, some parameters cannot be constant as the solutions to seemingly appear in a high correlation with the data. We were aware of impracticality to solve a deterministic model many times that exhaust all trials of the parameters, or to run its stochastic version with Monte Carlo strategy that also appeals for a high number of solving processes. We were also aware that those aforementioned non-constant parameters can potentially have particular relationships with several extrinsic factors, such as meteorology and socioeconomics of the human population. We then study an estimation of time-space-varying parameters within the framework of variational calculus and investigate how some parameters are related to some extrinsic factors. Here, a metric between the aggregated solution of the model and the empirical data serves as the objective function, where all the involving state variables are kept satisfying the physical constraint described by the model. Numerical results for some examples with real data are shown and discussed in details.
~
\vspace*{-5pt}
\end{abstract}

%\tableofcontents

\section{Introduction}
\label{sec:intro}
Dengue is a mosquito-borne disease that has become a noteworthy global health threat over the course of past few decades. The coverage of where it spreads includes tropical and sub-tropical regions, prevalently urban and peri-urban areas. Nowadays, a constant rise of rural economy and population, overwhelming flow of trades as well as tourisms combined with the global climate change have amplified the spread of dengue in rural areas. Also, the flow of migrants from regions of conflict might bear the risk of spreading the disease, cf.~\cite{SK2014i}. Recent findings indicated a positive trend of annual cases starting from approximately 0.4 to 1.3 million in 1996--2005, then growing dramatically to 2.2 million in 2010 and 3.2 million in 2015 \cite{WHO2012i}. In South-East Asia, for example, \textsc{WHO} closed its book every year with reports of annual cases from member countries, except the Democratic People's Republic of Korea, starting from 96,330 in 1986 to 232,530 in 2009 \cite{WHO2011i}. Among others, Thailand served as the most endemic region for dengue, followed by Indonesia. However, those all aforementioned numbers seemingly underestimated the real cases due to improper, different methods of recording -- a much larger portion of people might still be in probable cases of infection \cite{BGB2013i,SSU2016i}. 

Controlling dengue outbreaks is resource intensive and an unreasonable prediction of the incidences over the next time window may lead to a high inefficiency in the resource deployment. One classical approach attempts to generate a direct relationship between the historical incidences and time using a certain regressing function or a recursive relation, and the prediction is obtained by computing the most plausible values of the involving parameters. However, not only does this approach lose much of the physics behind the phenomenon that outputs the incidence data, but it also -- in general -- lacks the extent to which one should believe that it indeed governs the occurrence of the incidences. It is therefore important to derive a suitable mathematical model based on a prior understanding of such physics, which within the present context may include the relationship between the disease pattern and the dynamics of the mosquito population, the attribution of the interrelation between human and mosquito to the disease incidences, the influence of some extrinsic factors (for example meteorology and socioeconomics) to the interrelation, and so on. The more physical properties to be considered, the more complicated the model is, but the more reasonable the result would be.

The aim of this work was to present some parameter estimation over simple SIRS models. The optimal parameters are found in the sense that the metric between the model solution and incidence data achieves its minimum. We recall a standard SIRS model that is based on the work of \textsc{Kermack--McKendrick} in \cite{KM1927i}. One known drawback of this classical SIRS model is, as it merely averages the behavior of the population within a spatial region, that it does not capture the real mechanism of how the disease spreads from a locality through some more other localities. Of course, this spreading processes take some time and every locality might be supplemented by unique initial states at the beginning of the observation. It is also desirable to take the inhomogeneous population density into consideration to have a more reasonable model. 

We shall then extend the classical SIRS model to its reaction-diffusion version. It features basic characteristics that describe a large spectrum of biological phenomena, see \cite{Mur1989i}. The additional diffusion term can be a point of departure for modeling the situation where a group of infected individuals initially concentrates in a locality, then gradually spreads through a larger area with irregular motions. Another interesting feature of reaction-diffusion model is that its analytical properties are strongly dependent on the reaction term. Often, the existence and boundedness of a solution are settled if the invariant set of the spatially homogeneous system is bounded, which cannot be the case for some complicated reaction terms. However, we do not lay novelty in the modeling perspective. Many previous publications have dealt with some more enrichments, including sophisticated functional response in the infection term \cite{LMH2014i} and cross-diffusions \cite{OL2001i,MK1980i,MY1982i,GGJ2001i,BL2010i} to take into consideration the fact that a healthy person has more probability of getting infected if (s)he is spatially close to the infected people than the other healthy people.  

Here, we merely use a toy model to investigate the time and spatial dependencies of some parameters with which the corresponding unique aggregated solution tracks given epidemic data. It is considered very important to draw such a leading pattern of the endemicity to take appropriate precautions and deploy optimal health measures that help diminish the effects of the disease in the future. The rationale of using time-space-varying parameters is that the solution can track the data more accurately as compared to applying constant parameters to the same model and to that with additional noise.

We use the data of hospitalized cases from the city of Semarang, Indonesia for the test problem, which can be counted as a source of novelty in this paper. We highlight the fact that the immature phases of mosquito are aquatic and some of them may sometimes live in some sites where meteorological parameters fluctuate. Therefore, the dynamics of mosquito is also dependent on weather. This is a point where meteorology indirectly affects the infection rate.

\section{Problem formulation}
\label{sec:problem}
Let us consider an open spatial domain $\Omega\subset\mathbb{R}^2$, which is assumed to be bounded \textsc{Lipschitz}ian, meaning that the boundary $\partial\Omega$ can be locally graphed by  \textsc{Lipschitz} continuous functions. For the sake of notational convenience, we use the $y_1y_2y_3$--notation as a replacement for the standard SIR--notation. Let $y=y(t,x)$ denote the density of the total population at time $t\in [0,T)$ and spatial location $x\in \Omega$, comprising the density of susceptible people $y_1=y_1(t,x)$, the density of infected people $y_2=y_2(t,x)$, and that of recovered people $y_3=y_3(t,x)$. Note that $T>0$ is a fixed endpoint of the observation period. To better understand the biological meaning of the model, we have imposed several assumptions:
\begin{enumerate}[label=\normalfont (A\arabic*)]
\item Any newborn is always susceptible and the population is homogeneous in the sense that each individual has the same right to contribute to the growth as well as the decline of the population.\label{asm:newborn} 
\item The number of incidences is equivalent to the number of \emph{Aedes aegypti} mosquitoes. Meanwhile according to \textsc{Gubler} and \textsc{Meltzer} in \cite{GM1999i}, the number of incidences is equivalent to that of humans. We then exclude the dynamics of mosquito in the model due to this transitiveness. Therefore, the rate at which the disease propagates is simply proportional to the product of susceptible and infected people. Here, the proportionality varies in time and space determined by the infection rate $\beta=\beta(t,x)$. \label{asm:beta}
\item The growth as well as the death rates are proportional to the population size -- the model has to be derived in such a way that the total population throughout the whole domain is constant. \label{asm:constpop}
\item Any susceptible person experiences a very short incubation period, allowing her (him) to clinically move to the infected stage immediately after the infection. \label{asm:zeroincub}
\item There are no migrations out of and into the domain, i.e. the model exhibits the flux-free boundary conditions. \label{asm:fluxfree} 
\item A simple consideration to what is called \emph{temporary cross immunity}, cf. \cite{ABK2011i}, is reflected from letting some portion of the recovered people to move back to the susceptible stage. \label{asm:ade} 
\end{enumerate}
Considering all the given assumptions, we derive a reaction-diffusion model:
\begin{subequations}\label{eq:model}
\begin{alignat}{2}
	\partial_t y_1-d_1\Delta y_1&=\mu (y-y_1) - \beta\frac{y_1y_2}{y} + \kappa y_3\quad &&\text{in }Q :=  (0,T) \times \Omega,\label{eq:S}\\
	\partial_t y_2-d_2\Delta y_2&=\beta\frac{y_1y_2}{y} - (\gamma+\mu) y_2\quad &&\text{in }Q,\label{eq:I}\\
	\partial_t y_3-d_3\Delta y_3&= \gamma y_2 - (\kappa+\mu) y_3\quad &&\text{in }Q,\label{eq:R}\\ 
	0&=\partial_{\nu} y_1=\partial_{\nu} y_2=\partial_{\nu} y_3\quad &&\text{on }\Sigma :=  (0,T) \times \partial\Omega,\label{eq:flux}\\ 
	y_1&=y_{1,0},\,y_2=y_{2,0},\,y_3=y_{3,0}\quad &&\text{in }Q_0 :=  \{0\}\times \Omega.\label{eq:init}
\end{alignat}
\end{subequations}
We denote by $d_1,d_2,d_3$ the diffusion coefficients, by $\mu,\beta,\gamma,\kappa$ the death rate, the infection rate, the recovery rate, and the rate of transition indicating the simplest representation of \emph{temporary cross immunity}, respectively. We assume that all those parameters are positive. In the virtue of information sharing among experts, those constant parameters can somehow be defined as to lie within a set of feasible values. The first line \eqref{eq:S} indicates the assumptions~\ref{asm:newborn} and \ref{asm:beta}. Another technical motivation behind imposing \ref{asm:beta} is twofold: that there is no data of mosquito population at least in the region of observation, and that after time scale separation in a typical host-vector model under constant mosquito population, the dynamics of mosquito population at a faster time scale goes towards an equilibrium where a close relationship between the seasonal variation of mosquito and that of human population can explicitly be inferred, see for the details in \cite{RAS2013i,GAB2017i}. The second and third lines \eqref{eq:I}--\eqref{eq:R} indicate the assumptions~\ref{asm:zeroincub} and \ref{asm:ade} respectively. The fourth line \eqref{eq:flux} is the homogeneous \textsc{Neumann} boundary condition, given $\nu$ the outward unit normal to $\Omega$ along its boundary $\partial \Omega$, which is nothing else but the assumption~\ref{asm:fluxfree}. Furthermore, summing up \eqref{eq:S}--\eqref{eq:R} and taking the integration over $\Omega$ return in the assumption~\ref{asm:constpop}. The last line \eqref{eq:init} indicates the initial condition. For the sake of simplicity, we rewrite as $f_1,f_2,f_3$ the reaction terms that associate with $y_1,y_2,y_3$, respectively.

%For brevity, the model \eqref{eq:model} can be folded into the following form
%\begin{subequations}\label{eq:model}
%\begin{alignat}{2}
%	\frac{\partial y}{\partial t}-\nabla\cdot D\nabla y&=f(t,x,y)\quad &&\text{in }Q,\label{eq:y}\\ 
%	D\frac{\partial y}{\partial n}&=0\quad &&\text{on }\Sigma,\label{eq:flux}\\ 
%	y&=y_0\quad &&\text{in }Q_0,\label{eq:init}
%\end{alignat}
%\end{subequations}
%by $y :=  (S,I,R)$ and $D :=  \diag(d_S,d_I,d_R)$.
%The vector field $f$ is measurable over $Q$ for any fixed $u\in\mathbb{R}^n$ and affine over the non-constant parameters in $\zeta$.

Suppose that only a dataset of hospitalized cases $y_2^{d}$ is available. A first technical issue arises as some data points might not lie exactly at the equidistant spatial grid-points specified by a uniform discretization over the system \eqref{eq:model} or, moreover, exhibit \emph{granulations} (partial densenesses). To this unfavorable case we have to interpolate and extrapolate the data points. Now, for the only sake of mathematical convenience, we let $y_2^d$ inter- and extrapolated to occupy the whole domain $[0,T)\times \Omega$. Our main goal in this article is to solve the following fitting problem
\begin{align}\label{eq:optee}
	&\min_{\beta\in L^{\infty}(0,T)}J(y_1,y_2,y_3,\beta):= \intt_{Q} \frac{1}{2}\bra{y_2-y_2^d}^2\,\dd x\,\dd t+\frac{\omega}{2}\intt_{Q} \beta^2\,\dd x\,\dd t\nonumber\\
	&\text{subject to }\beta_{\min}\stackrel{\aev}{\leq}\beta\stackrel{\aev}{\leq} \beta_{\max}\text{ and }\eqref{eq:model}. 
\end{align} 
Minimizing the objective functional $J$ reflects our aim that, not only does $\beta$ have to lead the model solution $y_2$ to be as close to the given data $y_2^d$ as possible, but it also cannot be arbitrarily large. The \emph{regularization} parameter $\omega>0$ in the second term accounts for managing the trade-off between these conflicting criteria. The appearing quadratic terms in the objective functional also represent our concern to regularity of the solution.

If the diffusion parameters $d_1=d_2=d_3=0$, then we call the remaining system in \eqref{eq:model} as a \emph{spatially homogeneous system}, otherwise a \emph{spatially inhomogeneous system}. Vanishing $d_1,d_2,d_3$ means omitting the dispersal of individuals. Accordingly, an individual that was initially set to stay at a specific location $x$ will remain exactly at $x$ during the observation period $(0,T)$. This spatially homogeneous model may sound unrealistic from the application point of view, but asymptotic analysis of such scenario is considered much easier than that of the spatially inhomogeneous counterpart. In case of a constant infection rate $\beta$, the equilibria remain the same throughout $\Omega$. Another aim in this article is to see how the solution of the spatially inhomogeneous system behaves around the equilibria of the spatially homogeneous system. This is a specific issue where the so-called \textsc{Turing} instability analysis comes into play. 

Sometimes we are supplied with data that reflects the accumulation of incidences throughout $\Omega$ from time to time -- no more information regarding at which part of $\Omega$ gives what percentage in the accumulation. Unfortunately, integrating the system \eqref{eq:model} over $\Omega$ does not lead to a system of \textsc{ODE}s due to nonlinearity of the reaction term, i.e. the \textsc{H\"{o}lder} inequality might apply.  We propose to set the classical SIRS model adopted from the spatially homogeneous version of \eqref{eq:model} as to approach this type of data. Not only does this approach give an approximate value of the recovery rate $\gamma$, but it also helps to shed light on the typical procedures of parameter estimation for the \textsc{PDE} model. In case of a constant $\beta$, the problem reduces to the classical parameter estimation, which is mostly expressible as an inverse problem \cite{Boc1983i,WGS2015i}. 
%However, from analytical point of view, we cannot expect that both the PDE and ODE models characterize similar results only based on consideration that the diffusion part merely takes a role in the smoothing process. For general reaction terms and some extent of parameters, e.g. in \cite{GK1998}, the smoothing is hardly the case. In many publications, optimal control problems distributed control or boundary control, now control serves as a parameter, making the analysis of existence of global optimizer a bit more sophisticated.

\section{Elementary analysis}
\label{sec:analysis}

Before touching on the existence of a nonnegative weak solution, we shall recall several basic notations. Let $\bra{X,\norm{\cdot}_X}$ and $U$ be a \textsc{Banach} space and an open subset of $\R^m$ respectively, given either $m=1$ or $m=2$. In case $m=2$, we let $U$ be an open bounded \textsc{Lipschitz}ian domain. By $L^p(U)$, given $1\leq p\leq \infty$, we denote the space of all \textsc{Lebesgue} measurable functions $y:U\rightarrow \R$ satisfying $\norm{y}_{L^p(U)}<\infty$. The aforementioned norm is defined as $
\bra{\int_{U}\abs{y}^p\,\dd x}^{1\slash p}$ for $1\leq p<\infty$ and $\inf_{\abs{E}=0}\sup_{U\backslash E}\abs{y}$ for $p=\infty$. The \textsc{Sobolev} space of all functions $y\in L^p(U)$ with all the weak derivatives $D^{\alpha}y$ in $L^p(U)$ for all $\abs{\alpha}\leq k$ is denoted by $W^{k,p}(U)$. In a particular case for $p=2$, one often rewrites $W^{k,2}(U)=:H^k(U)$. The space $C^k(U;X)$ denotes the space of all functions $y:U\rightarrow X$ where $D^{\alpha}y$ for all $\abs{\alpha}\leq k$ are continuous on $U$. Particularly, $C^{\infty}(U)=\bigcap_{k=0}^{\infty}C^k(U)$ denotes the space of infinitely differentiable functions on $U$. Furthermore, we denote by $\bra{\cdot,\cdot}_H$ an inner product that endows a \textsc{Hilbert} space $H$, also by $y^+=\max(y,0)$ the positive part $y$.

We consider the evolution triple $V\hookrightarrow H\cong H^{\ast}\hookrightarrow V^{\ast}$ by identifying $H$ as a separable \textsc{Hilbert} space together with its dual $H^{\ast}$ via the \textsc{Riesz} isomorphism $h^{\ast}\sim \inner{h,\cdot}$, where $\inner{\cdot,\cdot}$ denotes the pairing between $H$ and $H^{\ast}$. The space $V$ is a reflexive \textsc{Banach} space such that it is compactly embedded into $H$, and $H$ is continuously embedded into $V^{\ast}$. In this paper we use $V=H^1(\Omega)$ and $H=L^2(\Omega)$, given $\Omega$ the domain of interest. We aim at guaranteeing the existence of a nonnegative weak solution of our model \eqref{eq:model} given the initial condition $y_{i,0}\in H$ for $i=1,2,3$. The notion of weak solution appears because $y_{i,0}\in H$ does not guarantee the existence of a classical solution of the model~\eqref{eq:model} throughout $\Omega$ from time to time. Furthermore, the mapping $V\rightarrow V:v\mapsto v^{+}$ is bounded, therefore $y_i\in W(0,T):=\{y\in L^2(0,T;V):\partial_t y\in L^2(0,T;V^{\ast})\}$ implies $y_i^+\in L^2(0,T;V)$. In addition to all biological assumptions \ref{asm:newborn}--\ref{asm:ade}, we extend the definition of the reaction term $(f_1,f_2,f_3)$ of the model~\eqref{eq:model} by giving the value zero to $y_1y_2\slash y$ whenever $y_1=0$ or $y_2=0$. Consequently, $(f_1,f_2,f_3)$ becomes continuously differentiable, and moreover, \textsc{Lipschitz} continuous.

\textbf{Weak solution}. A vector-valued function $(y_1,y_2,y_3)$ is called a \emph{weak solution} of the problem~\eqref{eq:model} if $y_i\in W(0,T)$ and
\begin{equation*}
\int_{\Omega}\partial_ty_i\,\varphi_i\,\dd x+d_i\int_{\Omega}\nabla y_i\cdot\nabla \varphi_i\,\dd x=\int_{\Omega}f_i\varphi_i\,\dd x
\end{equation*}
for almost all $t\in [0,T)$ and all $\varphi_i\in L^2(0,T;V)$, while it holds $y_i(0,\cdot)=y_{i,0}$ for $i=1,2,3$.

%Note that this space follows the embedding $W(0,T)\hookrightarrow C([0,T];L^2(\Omega))$, see for example \cite[Theorem~2, p.302]{Eva2010}. In return to the Lipschitz continuity of $f$, there are at least three characteristics of $f$ which we can benefit for further analysis:
%\begin{enumerate}[label=\normalfont (B\arabic*)]
%\item $f$ is uniformly bounded for all $(t,x)\in Q$, i.e. there exists $M$ such that $\lVert f(t,x,0)\rVert \leq M$.
%\item $f$ is globally Lipschitz continuous with respect to $y$ for all $(t,x)\in Q$. This holds by a prior knowledge that the function was continuously differentiable.
%\item Therefore, $f$ is measurable on $Q$ for any fixed value of $y\in \mathbb{R}^3$.
%\end{enumerate}

\begin{theorem}\label{thm:existee}
Let $\beta\in L^{\infty}(0,T)$ such that $\beta_{\min}\stackrel{\aev}{\leq}\beta\stackrel{\aev}{\leq} \beta_{\max}$ with $\beta_{\min},\beta_{\max}$ being in $L^{\infty}(0,T)$ and $\beta_{\min}\geq 0$. Then, the model \eqref{eq:model} admits a bounded weak solution. 
\end{theorem}

Classical approaches for the proof are twofold: \textsc{Galerkin} approximation \cite{Zei1990i,Wlo1987i} and semi-group approach \cite{HP1957i}. However, their precise presentation usually incorporates many technical details. Here, we  put forward the only outline of the classical \textsc{Galerkin} approximation along with several supporting notions  for keeping the reader on track. Similar and more detailed proofs can be looked at e.g. \cite[Theorem~7.8, pp.367--372]{Tro2010i} and \cite{BL2010i}, also that for the problem with \textsc{Dirichlet} boundary condition at \cite[Theorem~3, p.378--379]{Eva2010i}. 

\begin{proof}[Extended sketch for the proof of Theorem~\ref{thm:existee}]
Firstly, we solve the spectral problem as a weak formulation of the eigenvalue problem $-\Delta \varphi=\lambda\varphi$ in $\Omega$ via test functions in $V$, equipped by the homogeneous \textsc{Neumann} boundary condition. Without loss of generality, we may define $\Omega=(0,a)\times (0,b)$. The corresponding eigenvalues as well as the eigenvectors $(\lambda_l,\varphi_l)_{l\geq0}$ are given by 
\begin{align*}
\lambda_l &= \lambda_{j,k}=\pi^2\left(\left(\frac{j}{a}\right)^2+\left(\frac{k}{b}\right)^2\right),\\
\varphi_l(x) &=\varphi_{j,k}(x_1,x_2)=\frac{jk\pi^2}{ab\sin(j\pi)\sin(k\pi)}\cos\left(\frac{j\pi x_1}{a}\right)\cos\left(\frac{k\pi x_2}{b}\right),
\end{align*}
where $j,k=0,1,2,\cdots$, excluding $j=k=0$. Now, one has to assign a careful enumeration. In this case, $j$ and $k$ are organized in such a way that $(\lambda_l)_{l\geq 0}$ is increasing i.e. by unique representation of eigenvalues. As for the trivial eigenvalue $\lambda_0=0$, i.e. where $j=k=0$, the corresponding eigenfunction $\varphi_0$ is known to be constant due to the \textsc{Neumann} boundary condition. Since $\bra{ \varphi_i,\varphi_j}_{L^2(\Omega)}=\delta_{ij}$, where $\delta_{ij}$ denotes the usual \textsc{Kronecker} delta, all these eigenfunctions $(\varphi_l)_{l\geq 0}$ are complete in the $L^2$ sense. 

Secondly, we check if the approximate solution using a finite set of basis eigenfunctions, i.e. $y_{i,n}(t,x)=\sum_{l=0}^ns_{i,n,l}(t)\varphi_l(x)$ for $i=1,2,3$, solves the problem~\eqref{eq:model} in the weak sense. This requires $y_{i,n}$ or $s_{i,n,l}$ to not blow up for all $n\geq 0$. Therefore, a check on regularity of $s_{i,n,l}$ is one of the important key steps in the proof. To see this, we use the test functions $\varphi_k$ for $k\in\set{0,\cdots,n}$ in the weak formulation of~\eqref{eq:model} with $y_i$ replaced by $y_{i,n}$ to obtain the \textsc{ODE} for $s_{i,n,k}$:
\begin{equation}\label{eq:ss}
\dd_ts_{i,n,k} = -d_i\int_{\Omega}\nabla y_{i,n}\cdot\nabla\varphi_k\,\dd x-\int_{\Omega}f_i(y_{1,n},y_{2,n},y_{3,n})\varphi_k\,\dd x.
\end{equation}
The last attempt is well-defined since $\bra{\varphi_l}_{l\geq 0}$ also sets an orthogonal basis in $V$. The initial condition is set up by defining $y_{i,n,0}\rightarrow y_{i,0}$ as $n\rightarrow\infty$, i.e. $s_{i,n,k}(0)=\bra{y_{i,n,0},\varphi_k}_{L^2(\Omega)}$. The fact that the right--hand side (rhs) of \eqref{eq:ss} is affine linear with respect to $(s_{i,n,l})_{i,l=0,\cdots,n}$ allows the standard \textsc{Peano} Theorem to guarantee the existence of a solution on a subset, say $[0,T_-)\subset [0,T)$. Moreover, using $y_{i,n}$ as the test function in the weak formulation returns in the boundedness of: (1) $\norm{\partial_ty_{i,n}}_{L^2(0,T_-)}$, (2) $\norm{y_{i,n}}_{L^2(0,T_-)}$, (3) $\norm{\nabla y_{i,n}}_{L^2([0,T_-)\times\Omega)}$, (4) the reaction terms $\norm{f_i(y_{1,n},y_{2,n},y_{3,n})}_{L^2([0,T_-)\times\Omega)}$ by $\beta_{\min}\stackrel{\aev}{\leq}\beta\stackrel{\aev}{\leq}\beta_{\max}$ and (5) the rhs of \eqref{eq:ss} with respect to $y_{1,n},y_{2,n}$ and $y_{3,n}$. Due to the boundedness and affine linearity of the rhs of \eqref{eq:ss} with respect to $(s_{i,n,l})_{i,l=0,\cdots,n}$ (\textsc{Carath\'{e}odory}), the standard Extension Theorem for \textsc{ODE}s guarantees the existence of continuous $(s_{i,n,k})_{k=0}^n$ solving \eqref{eq:ss} on a larger domain $[0,T_-+\delta)$, $\delta>0$. This process can be repeated over and over to guarantee the existence of continuous $(s_{i,n,k})_{k=0}^n$ on $[0,T)$ as well as the boundedness of the aforementioned five items. 

Thirdly, within bounded sequences $(y_{i,n})_{n\geq 0}$ in $L^2\bra{0,T;V}$ and $(\partial_ty_{i,n})_{n\geq 0}$ in $L^2\bra{0,T;V^{\ast}}$, we are concentrating on one limiting functional $y_{i}$ to which a sub-sequence $(y_{i,n_j})_{n_j\geq0}$ in $(y_{i,n})_{n\geq 0}$ converges. Replacing $n_j$ with $n$, a standard convergence theorem for reflexive \textsc{Banach} spaces (see e.g. \cite[pp.723--724]{Eva2010i}) guarantees that $y_{i,n}\rightharpoonup y_i$, $\partial_ty_{i,n}\rightharpoonup \partial y_i$, $\nabla y_{i,n}\rightharpoonup \nabla y_{i}$ in $L^2(0,T)$, $L^2(0,T)$, and $L^2(Q)$ respectively. By means of \textsc{Aubin--Lions} Lemma, $W(0,T)$ is compactly embedded in $L^2(0,T;H)$, leading to the strong convergence $y_{i,n}\rightarrow y_i$ in $L^2(Q)$. Then, it also holds $f_i(y_{1,n},y_{2,n},y_{3,n})\rightarrow f_i(y_1,y_2,y_3)$ in $L^2(Q)$ due to the \textsc{Lipschitz} continuity. 

Finally, using any arbitrary test function $\varphi_i\in L^2(0,T;V)$ in the weak formulation and setting $n\rightarrow \infty$, we discover that the limiting functional $y_{i}$ is a weak solution of \eqref{eq:model} for $i=1,2,3$. Boundedness follows from the weak convergence of $(y_{i,n})_{n\geq 0}$ and $(\partial_ty_{i,n})_{n\geq 0}$, which guarantees $(y_i)_{i=1}^3$ bounded in $L^2(0,T;V)$. Moreover, a similar idea as in \cite[p.379]{Eva2010i} that was also presented in \cite[Lemma~3.6]{BL2010i} settles the specific initial value problem if $y_i(0,\cdot)=y_{i,0}$ as given in \eqref{eq:init}.
\end{proof}

\begin{theorem}\label{thm:positiveee}
Let $\beta\in L^{\infty}(0,T)$ such that $\beta_{\min}\stackrel{\aev}{\leq}\beta\stackrel{\aev}{\leq} \beta_{\max}$ where $\beta_{\min}\geq 0$ and the initial condition $y_{i,0}\in L^2(\Omega)$ be nonnegative for $i=1,2,3$. Then, the existing weak solution of \eqref{eq:model} is unique and nonnegative.
\end{theorem}
\begin{proof}
Uniqueness follows from the idea of proving that the solution of \eqref{eq:model} for $(f_1,f_2,f_3)=0$ and $(y_{1,0},y_{2,0},y_{3,0})=0$ is $0$. Setting $\varphi_i:=y_i\in W(0,T)$ as the test function in the weak formulation, we have in the sum
\begin{align*}
\sum_{i=1}^3\dd_t\frac{\norm{y_i}^2_H}{2}+\sum_{i=1}^3d_i\int_{\Omega}\abs{\nabla y_i}^2\,\dd x=0,
\end{align*}
which gives us $\dd_t\norm{y_i}^2_H\leq 0$ for $i=1,2,3$. Then, the zero solution is settled by considering the zero initial condition.

As for nonnegativity, we follow the basic idea from \textsc{Dautray} and \textsc{Lions} in \cite[Theorem~2, p.534]{DL1992i} and use some brief exploration of the positive part of a function as in~\cite{Wac2016i}. Given $y_i\in W(0,T)$, we know that $y_i^+\in L^2(0,T;V)$. Owing to the denseness of $C^{\infty}(0,T;V)$ in $W(0,T)$ we guarantee the existence of a sequence $(y_{i,n})_{n\geq 0}\in C^{\infty}(0,T;V)$ such that $y_{i,n}\rightarrow y_i\in W(0,T)$. Moreover, the mapping $y_i\mapsto y_i^+$ is continuous in $y_i$, therefore $y_{i,n}^+\rightarrow y_i^+$ in $C(0,T;H)$. By definition, the sequence $(y^+_{i,n})_{n\geq 0}$ is bounded in $L^2(0,T;V)$, therefore $y^+_{i,n_j}\rightharpoonup y^+_i$ in $L^2(0,T;V)$ for some sub-sequence $(y^+_{i,n_j})_{n_j\geq 0}$. Since $y_{i,n}^+\rightarrow y_i^+$ in $C(0,T;H)$, then the whole sequence $y_{i,n}^+\rightharpoonup y_i^+$ in $L^2(0,T;V)$. We also know that $(\partial_ty_{i,n})_{n\geq 0}\in C^{\infty}(0,T;V^{\ast})$. Then, the similar way as before leads us to the conclusion that $\partial_ty^+_{i,n}\rightharpoonup \partial_ty^+_{i}$ in $L^2(0,T;V^{\ast})$. Using the boundedness of both the function and derivative and assigning the test function $y_{i,n}^+$ in the weak formulation with $y_i$ replaced by $y_{i,n}$, we also obtained the boundedness of $\nabla y_{i,n}^+$ in $L^2(Q)$. Consequently, $\nabla y_{i,n}^+\rightharpoonup \nabla y_{i}^+$ in $L^2(Q)$. 

The negative part $y_i^-$ is nothing else but $(-y_i)^+$. At the expense of using the test function $(-y_{i,n}^-)$ in the weak formulation we obtain
\begin{align*}
\int_{\Omega}\partial_t y_{i,n}(-y_{i,n}^-)\,\dd x&=\int_{\Omega}\partial_ty_{i,n}(-(-y_{i,n})^+)\,\dd x=\int_{\Omega}\partial_t(-y_{i,n})(-y_{i,n})^+\,\dd x\\
&=\int_{\Omega}\partial_t(-y_{i,n})^+(-y_{i,n})^+\,\dd x=\dd_t\int_{\Omega}\frac{\abs{(-y_{i,n})^+}^2}{2}\,\dd x\\
&=\dd_t\frac{\norm{y_{i,n}^-}^2_H}{2}.
\end{align*}
Taking $n\rightarrow \infty$, it holds $\int_{\Omega}\partial_t y_{i}(-y_{i}^-)\,\dd x=\dd_t\norm{y_{i}^-}^2_H\slash 2$. Now set $\varphi_{i,n}:=-y_{i,n}^-$ as a test function, which clearly lies in $L^2(0,T;V)\cap C(0,T;H)$. Multiplying each of the handling equations in \eqref{eq:model} with $\varphi_{i,n}$ where $y_{i}$ is replaced by $y_{i,n}$, integrating over $\Omega$, then taking $n\rightarrow \infty$ we obtain in the sum
\begin{align*}
&\sum_{i=1}^3\dd_t\frac{\norm{y_i^-}^2_H}{2}+\sum_{i=1}^3d_i\int_{\Omega}\abs{\nabla y_i^-}^2\,\dd x\\
&=-(\gamma+\mu)\int_{\Omega}\abs{y_2^-}^2\,\dd x-(\mu+\kappa)\int_{\Omega}\abs{y_3^-}^2\,\dd x\leq0,
\end{align*}
which gives us $\sum_{i=1}^3\dd_t\frac{\norm{y_i^-}^2_H}{2}\leq0$. The specified initial condition supplements with $y_i^-(0,\cdot)=0$ for all $i=1,2,3$, leading to $y_i^-(t,\cdot)=0$ for $t> 0$. 
\end{proof}

By the maximum principle semilinear parabolic systems (see for similar discussions in \cite{Hen1981i,PW1984i,BL2010i} and references therein) and \textsc{Lipschitz} continuity of the reaction term, for any given $d_1,d_2,d_3>0$ and nonnegative $(y_{1,0},y_{2,0},y_{3,0})\in \left(C^{2+r}(\overline{\Omega})\right)^3$ where $r\in (0,1)$ there exists an $M<\infty$ such that the system \eqref{eq:model} has a unique classical solution $(y_1,y_2,y_3)\in \left(C^{\frac{2+r}{2},2+r}([0,\infty)\times\overline{\Omega})\right)^3$ satisfying 
\begin{equation*}
0\leq y_1(t,x),y_2(t,x),y_3(t,x)\leq M
\end{equation*}
for all $t>0$ and $x\in\overline{\Omega}$.

\subsection{\textsc{Turing} instability analysis}
Let us from now on denote $Y:=(y_1,y_2,y_3)$, $F:=(f_1,f_2,f_3)$, and $D=\diag\bra{d_1,d_2,d_3}$. Consider the spatially homogeneous version of \eqref{eq:model}
\begin{equation}\label{eq:homogen}
\frac{\partial Y}{\partial t}=F(Y),
\end{equation}
where all the parameters are assumed to be constant. For the sake of well-posedness of both spatially homogeneous and inhomogeneous systems, let us assume that the initial condition $Y_0$ is smooth enough (at least $C^2$) on $\overline{\Omega}$ and a classical solution for the spatially inhomogeneous system exists on $(0,\infty]$. Since $F$ is autonomous, the equilibria of \eqref{eq:homogen} can be derived by the straightforward calculation, returning the \emph{disease-free equilibrium} and the \emph{endemic equilibrium}
\begin{align}
Y^D& :=  (y_0,0,0),\label{eq:dfe}\\
Y^E& :=  \left(\underbrace{y_0\frac{(\gamma+\mu)}{\beta}}_{y_1^E},\underbrace{y_0\frac{(\kappa+\mu)(\beta-\gamma-\mu)}{\beta(\gamma+\mu+\kappa)}}_{y_2^E},\underbrace{y_0\frac{\gamma(\beta-\gamma-\mu)}{\beta(\gamma+\mu+\kappa)}}_{y_3^E}\right),\label{eq:end}
\end{align}
where $y_0=y_1^E+y_2^E+y_3^E=y(0,x)$ for any $x\in\Omega$. 

We attempt to study, under what conditions the aforementioned equilibria are locally (asymptotically) stable in the sense that the solutions for the non-homogeneous system stay close (converge) to them in the long run. As for the moment, we let $Y^S$ be a locally stable equilibrium of the system \eqref{eq:homogen} and initiate some abstraction. The \textsc{Turing} instability analysis aims to study how stable is the solution $Y$ of the spatially inhomogeneous system \eqref{eq:model} relative to $Y^S$. We denote by $Z :=  Y-Y^S$ the corresponding error between both terms. Moreover, $Z$ obeys the following equation in the linearization mode
\begin{equation*}
\partial_t Z - D\Delta Z = \partial_Y F\bra{Y^S}Z
\end{equation*}
together with the usual homogeneous \textsc{Neumann} boundary condition. Owing to the previous result on the eigenvalue problem $-\Delta\varphi=\lambda\varphi$, we know that the eigenfunctions $(\varphi_l)_{l\geq 0}$ sets an orthonormal basis in $L^2(\Omega)$ and an orthogonal basis in $V$. Therefore, for any $Z\in L^2(0,T)$, there exist $(s_l)_{l\geq0}$ such that $Z(t,x)=\sum_{l=0}^{\infty}s_l(t)\varphi_l(x)$. Putting back this specification to the very last equation of $Z$ we obtain the linear equations
\begin{equation*}
\dd_t s_l=\left(-\lambda_lD+\partial_Y F\bra{Y^S}\right)s_l=: A_ls_l,\quad l=0,1,2,\cdots
\end{equation*}
By the \textsc{Lyapunov} stability of the zero error $Z=0$ we have to ensure that all $(s_l)_{l\geq0}$ converges to constants nearby zero as $t\rightarrow\infty$, where it can happen if and only if all eigenvalues of $A_0,A_1,A_2,\cdots$ have non-positive real part.  

\textbf{Disease-free equilibrium}. The \textsc{Jacobi}an of the vector field $F$ evaluated at the disease-free equilibrium $Y^D$ is given by
\begin{equation*}
\partial_Y F\bra{Y^D}=\left(\begin{array}{ccc}
-\mu & - \beta & \kappa\\
0 & \beta -\gamma-\mu & 0\\
0 & \gamma & -\kappa-\mu
\end{array}
\right)
\end{equation*}
whose eigenvalues are $-\mu,\beta-\mu-\gamma,-\kappa-\mu$. This finding suggests that $Y^D$ is locally asymptotically stable if $\beta-\gamma-\mu< 0$ or $\frac{\beta}{\gamma+\mu}< 1$. The eigenvalues of the matrix $A_l=-\lambda_l D+\partial_Y F\bra{Y^D}$ are given by
\begin{equation*}
-\lambda_ld_1-\mu,\,\beta-\gamma-\mu-\lambda_ld_2,\,\text{and }-\kappa-\mu-\lambda_ld_3.
\end{equation*}
As we keep the condition $\frac{\beta}{\gamma+\mu}< 1$, all the eigenvalues remain negative. We immediately arrive at the following summary.
\begin{lemma}
The solution of the spatially inhomogeneous system \eqref{eq:model} remains stable around the disease-free equilibrium $Y^D$ providing that $\frac{\beta}{\gamma+\mu}< 1$.
\end{lemma}
The parameter $\frac{\beta}{\gamma+\mu}$ is what is known from the spatially homogeneous SIRS model as the \emph{basic reproductive number}.

\textbf{Endemic equilibrium}. A glimpse over the formulation of the endemic equilibrium $Y^E$ in~\eqref{eq:end} shows that it is biologically meaningful as the basic reproductive number $\frac{\beta}{\gamma+\mu}> 1$. We henceforth keep this condition in mind. The \textsc{Jacobi}an of the vector field $F$ evaluated at $Y^E$ is given by
\begin{equation*}
\partial_Y F\bra{Y^E}=\left(\begin{array}{ccc}
-\mu-p & -\gamma-\mu & \kappa\\
p & 0 & 0\\
0 & \gamma & -\kappa-\mu
\end{array}
\right),
\end{equation*}
where $p=(\beta-\gamma-\mu) (\kappa+\mu)\slash (\gamma+\kappa+\mu)$. The characteristic polynomial of the \textsc{Jacobi}an can immediately be derived as
\begin{equation*}
\p(\eta) = {\eta}^{3}+ \left( \kappa+2\mu+p \right) {\eta}^{2}+ \left( \gamma
p+\kappa\mu+p\kappa+{\mu}^{2}+2p\mu \right) \eta+p\mu \left( 
\gamma+\kappa+\mu \right).
\end{equation*}
One sufficient condition such that $Y^E$ is locally asymptotically stable is then $p>0$, which agrees with $\frac{\beta}{\gamma+\mu}> 1$. Therefore no further any condition is concerned at this point. Without writing in detail, the characteristic polynomial of the matrix $A_l=-\lambda_l D+\partial_Y F\bra{Y^E}$ shows to have positive coefficients providing that $p>0$. Therefore, we have the following summary.
\begin{lemma}
The solution of the spatially inhomogeneous system \eqref{eq:model} remains stable around the endemic equilibrium $Y^E$ providing that the basic reproductive number $\frac{\beta}{\gamma+\mu}> 1$.
\end{lemma}
\begin{remark}
By $\frac{\beta}{\gamma+\mu}=1$, both the \textsc{Jacobi}ans eveluated at the aforementioned equilibria contain zero eigenvalue. The task remains to find out which one among the dynamics associated with the equilibria on the local center manifolds is stable, since then the corresponding equilibrium becomes locally \textsc{Lyapunov} stable. This task first requires transformation of the original system around the investigated equilibrium for which the concept of local center manifold can be used. The corresponding discussion can be technically long but of less significant importance which we therefore omit in the presentation.   
\end{remark}

\section{Existence of optimal parameter}
\label{sec:optparam}
In the sequel, we discuss the well-posedness of the optimization problem \eqref{eq:optee} in the sense that it admits a global minimizer of the objective functional. The argument for proving the existence in this paper is inspired by \textsc{Tr\"{o}ltzsch} in \cite[Chapter~4.4 and 5.3]{Tro2010i}. 

\begin{theorem}\label{thm:existeeoptee}
To the problem~\eqref{eq:optee} there exists at least one globally optimal minimizer $\beta\in L^{\infty}(0,T)$.
\end{theorem}
\begin{proof}
Given $\beta\in \mathcal{B}:=\set{\alpha\in L^{\infty}(0,T):\,\beta_{\min}\stackrel{\aev}{\leq}\alpha\stackrel{\aev}{\leq}\beta_{\min}}$, we can compute $Y(\beta):=(y_1,y_2,y_3)(\beta)$ the unique weak solution of \eqref{eq:model}. A glimpse over the formulation of $J$ (convex and continuous) suggests that $J$ is bounded from below by zero. Therefore, there must exist the infimum $0\leq J_{\inf}:=\inf_{\beta\in \mathcal{B}}J(Y(\beta),\beta)$ and $(\beta_n)_{n\geq 0}\subset \mathcal{B}$ for which $\lim_{n\rightarrow\infty}J(Y_n,\beta_n)=J_{\inf}$, where $Y_n:=Y(\beta_n)$. We note that boundedness, convexity and closedness guarantee the weakly closedness of $\mathcal{B}$, i.e. there exists $\overline{\beta}\in\mathcal{B}$ for which $\beta_n\rightharpoonup\overline{\beta}$ in $L^2(0,T)$ for some sub-sequence $(\beta_n)_{n\geq0}$, i.e. at the expense of relabelling sub-indices with $n$. Consequently, there exists $\overline{Y}\in W(0,T)$ for which $Y_n\rightharpoonup \overline{Y}$. Two things remain to show: (1) $Y(\overline{\beta})=\overline{Y}$ and (2) $\overline{\beta}$ is a global minimizer of $J$.

By means of $(y_{i,n})_{n\geq0}\subset W(0,T)$ for $i=1,2,3$, the reaction terms $(F_n)_{n\geq0}$ where $F_n:=(f_1,f_2,f_3)(Y_n)$ are bounded in $L^2(Q)$ due to \textsc{Lipschitz} continuity. Consequently, there exists $(\overline{f}_1,\overline{f}_2,\overline{f}_3)=:\overline{F}\in L^2(Q)$ such that $F_n\rightharpoonup \overline{F}$. By weak solution, we obtain
\begin{align*}
\int_{\Omega}\partial_t y_{i,n}\varphi_i\,\dd x+d_i\int_{\Omega}\nabla y_{i,n}\cdot\nabla \varphi_i\,\dd x=\int_{\Omega}f_{i,n}\varphi_i\,\dd x\stackrel{n\rightarrow\infty}{\longrightarrow}&\\
\int_{\Omega}\partial \overline{y}_{i}\varphi_i\,\dd x+d_i\int_{\Omega}\nabla \overline{y}_{i}\cdot\nabla \varphi_i\,\dd x=\int_{\Omega}\overline{f}_i\varphi_i\,\dd x&
\end{align*}
for all $\varphi_i\in L^2(0,T;V)$ and $i=1,2,3$. It remains to show that $\overline{F}=F(\overline{Y})$, since then this equality together with $\beta_n\rightharpoonup \overline{\beta}$ and $Y_n\rightharpoonup \overline{Y}$ implies $Y(\overline{\beta})=\overline{Y}$, as desired. Due to \textsc{Aubin--Lions} Lemma, it holds that $Y_n\rightarrow \overline{Y}$ in $L^2(Q)$. Theorem~\ref{thm:existee} has confirmed that $(Y_n)_{n\geq 0}$ is bounded in $W(0,T)$, so is $\overline{Y}$. Employing the \textsc{Lipschitz} continuity of $F$, one obtains $\norm{F_n-F(\overline{Y})}_{L^2(Q)}\leq L\norm{Y_n-\overline{Y}}_{L^2(Q)}$ for some $L>0$. Taking $n\rightarrow \infty$, we immediately attain the strong convergence $F_n\rightarrow F(\overline{Y})$ in $L^{\infty}(Q)$ as well as $L^2(Q)$ by the boundedness of $(Y_n)_{n\geq 0}$ and convergence $Y_n\rightarrow \overline{Y}$. Adding to this from the previous result $F_n\rightharpoonup \overline{F}$ in $L^2(Q)$, it therefore holds $F(\overline{Y})=\overline{F}$ due to the fact that the limits in both topologies are the same.

As for proving that $\overline{\beta}$ is a global minimizer of $J$, we employ the continuity and convexity of $J_{\beta}$ in $J=J_Y+J_{\beta}$, where $J_{\beta}:=\frac{\omega}{2}\intt_{Q}\beta^2\,\dd x\dd t$, to have $\beta_n\rightharpoonup \overline{\beta}$ implying $\liminf_{n\rightarrow\infty}J_{\beta}(\beta_n)\geq J_{\beta}(\overline{\beta})$ according to \cite[Theorem~2.12, p.47]{Tro2010i} and specific references mentioned therein. Finally, 
\begin{equation*}
J_{\inf}=\lim_{n\rightarrow \infty}J(Y_n,\beta_n)=\lim_{n\rightarrow\infty}J_{Y}(Y_n)+\liminf_{n\rightarrow\infty}J_{\beta}(\beta_n)\geq J(\overline{Y},\overline{\beta}),
\end{equation*}
which shows the second claim. 
\end{proof}

\section{Numerical experiments}
\label{sec:numerics}
\subsection{Indirect method}

The indirect method is based on the principle "optimize then discretize". It aims at generating a governing system representing the first order necessary optimality conditions. This system consists of the state and adjoint equations, also the gradient of the objective functional that follows a certain inequality in the optimality. One method to solve this system is the gradient method \cite{Rag1977i}, which is based on iteratively seeking the optimal parameter $\beta$ in the respect of walking towards the infimum of $J$. 

Let us briefly derive the state and adjoint equations, taking a note that the gradient of the objective functional serves as the descent direction in the gradient method. Consider the \textsc{Lagrange}an function
\begin{align*}
\LM(Y,\beta,Z) = J(Y,\beta) + \sum_{i=1}^3\bra{f_i-\partial_t y_i,z_i}_{L^2(Q)}-\sum_{i=1}^3d_i\bra{\nabla y_i,\nabla z_i}_{L^2(Q)},
\end{align*}
where $Y:=(y_1,y_2,y_3)$ and $Z:=(z_1,z_2,z_3)$ for which $z_i\in W(0,T)$ denote the weak solution of \eqref{eq:model} for a given $\beta\in \mathcal{B}$ and the \textsc{Lagrange}an adjoint variables, respectively. The necessary optimality condition states that all the \textsc{G\^{a}teaux} directional derivatives (\textsc{G\^{a}teaux} variation) of $\LM$ with respect to all arguments towards all directions are zero. Zeroing the variation of $\LM$ with respect to $z_i$ towards all directions $h_z^i\in W(0,T)$ returns in the state equation \eqref{eq:S}--\eqref{eq:R}, i.e. by putting back $\int_{\Omega}d_i\nabla y_i\cdot\nabla z_i\,\dd x-d_i\int_{\partial\Omega}\nabla y_i\cdot \nu z_i\,\dd s(x)=-\int_{\Omega}d_i\Delta y_iz_i\,\dd x$. Zeroing the variation of $\LM$ with respect to $y_i$ towards all directions $h_y^i\in W(0,T)$, with $h_y^i(0,\cdot)=0$, and using integration by parts, we obtain
\begin{align*}
0&=\bra{\delta_{i2}(y_2-y_2^d),h_y^i}_{L^2(Q)}+\sum_{j=1}^3\bra{\partial_{y_i}f_jz_j,h_y^i}_{L^2(Q)}-\bra{h_y^i(T),z_i(T)}_{L^2(\Omega)}\\
&+\underbrace{\bra{h_y^i(0),z_i(0)}_{L^2(\Omega)}}_{=0}+\int_{(0,T)}\bra{\partial_tz_i,h_y^i}_{H^1(\Omega)}\,\dd t-d_i\bra{\nabla h_y^i,\nabla z_i}_{L^2(Q)}
\end{align*}
where $\delta_{i2}$ denotes the \textsc{Kronecker} delta. In the strong formulation, $z_i$ must now follow the reaction--diffusion characterizing adjoint equation
\begin{subequations}\label{eq:adjee}
\begin{alignat}{2}
	\partial_t z_i+d_i\Delta z_i&=-\sum_{j=1}^3\partial_{y_i}f_jz_j-\delta_{i2}(y_2-y_2^d)\quad &&\text{in }Q,\label{eq:zi}\\
	\partial_{\nu} z_i&=0\quad &&\text{on }\Sigma,\label{eq:fluxzi}\\ 
	z_i(T,\cdot)&=0\quad &&\text{in }Q_T:=\{T\}\times \Omega,\label{eq:initzi}
\end{alignat}
\end{subequations}
for all $i=1,2,3$. In an extended version, the adjoint system \eqref{eq:adjee} is given by
\begin{alignat*}{2}
	\partial_t z_1+d_1\Delta z_1&=\frac{\beta}{y^2}(y-y_1)y_2z_1-\frac{\beta}{y^2}(y-y_1)y_2z_2\quad &&\text{in }Q,\\
	\partial_t z_2+d_2\Delta z_2&=\bra{\frac{\beta}{y^2}(y-y_2)y_1-\mu}z_1-\bra{\frac{\beta}{y^2}(y-y_2)y_1-\gamma-\mu}z_2-\gamma z_3-(y_2-y_2^d)\quad &&\text{in }Q,\\
	\partial_t z_3+d_3\Delta z_3&= -\bra{\frac{\beta}{y^2}y_1y_2+\mu+\kappa}z_1+\frac{\beta}{y^2}y_1y_2z_2+(\mu+\kappa)z_3\quad &&\text{in }Q,\\ 
	0&=\partial_{\nu} z_1=\partial_{\nu} z_2=\partial_{\nu} z_3\quad &&\text{on }\Sigma,\\ 
	z_1&=z_2=z_3=0\quad &&\text{in }Q_T.
\end{alignat*}
Extending the space of the test functions $h_y^i$ to $L^2(0,T;V)$, we confirm the existence and uniqueness of a bounded weak solution $z_i$ for $i=1,2,3$ of \eqref{eq:zi}--\eqref{eq:initzi} by the virtue of Theorem~\ref{thm:existee} and Theorem~\ref{thm:positiveee}. To set up the initial condition in the proof, we have to first transform the time variable $\tau:=T-t$, and define new variables $\tilde{y}_i(\tau):=y_i(T-t)$, $\tilde{z}_i(\tau):=z_i(T-t)$, $\tilde{h}_y^i(\tau):=h_y^i(T-t)$, $\tilde{\beta}(\tau):=\beta(T-t)$ and $\tilde{f}_i:=f_i(\tilde{y}_1,\tilde{y}_2,\tilde{y}_3,\tilde{\beta})$.

Towards characterization of the gradient, it is essential to first check the differentiability of the weak solution $Y$ with respect to $\beta$. In the optimal control theory, the implicit representation of the (weak) solution in terms of control variable is often referred to as the \emph{control--to--state} operator. In our case, it is denoted by $S_i:\mathcal{B}\rightarrow W(0,T):\beta\mapsto y_i(\beta)$. We then rush through the following lemma before discussing the gradient method. Note that the proof is similar to that in \cite[Theorem~5.9, p.275]{Tro2010i}.

\begin{lemma}
The control--to--state operators $S_i$ for $i=1,2,3$ mapping from $L^{\infty}(0,T)$ to $W(0,T)\cap C(\bar{Q})$ are \textsc{Fr\'{e}chet} differentiable with respect to $\beta$. The corresponding derivative $\partial_{\beta}S_i$ is given by the unique solution of the following system
\begin{subequations}\label{eq:cts}
\begin{alignat}{2}
	\partial_t u_i-d_i\Delta u_i&=\sum_{j=1}^3\partial_{y_j}f_iu_j+\partial_{\beta}f_i\quad &&\text{in }Q,\label{eq:ui}\\
	\partial_{\nu} u_i&=0\quad &&\text{on }\Sigma,\label{eq:fluxui}\\ 
	u_i(0,\cdot)&=0\quad &&\text{in }Q_0,\label{eq:initui}
\end{alignat}
\end{subequations}  
for $i=1,2,3$ where $(y_1,y_2,y_3)$ is the weak solution of the state equation \eqref{eq:model}.
\end{lemma}
Then, the gradient method for solving \eqref{eq:optee} now reads as follows. We will use the notation $\PP_\mathcal{B}$ for the projection into $\mathcal{B}$, i.e. $\PP_{\mathcal{B}}(\alpha)=\arg\inf_{\beta\in\mathcal{B}}\norm{\alpha-\beta}_{L^2(Q)}$.

\begin{mdframed}
\begin{description}
\item[\textsc{Gradient method}]\label{alg:gradient}
\item[Return:] $Y^{(n)},\beta^{(n)},J^{(n)}$.
\item[Step 0:] Specify $n=0$, $\beta^{(n)}\in\mathcal{B}$, $\varepsilon>0$, $\delta_0>0$, $\dd J=\varepsilon+1$.
\item[Step 1:] Compute $Y^{(n)},Z^{(n)}$ from \eqref{eq:model} and \eqref{eq:adjee} as well as $J^{(n)}$ from \eqref{eq:optee}.
\item[Step 2:] While $|\dd J|\geq \varepsilon$, then do the following
\item[Step 2.a:] Compute $\partial_{\beta}Y$ from \eqref{eq:cts} and the gradient 
\begin{equation*}
\partial_{\beta}J^{(n)}=\partial_{\beta}J\bra{\beta^{(n)}}=\bra{\partial_{y_2}J\partial_{\beta}y_2+\partial_{\beta}J}\bra{\beta^{(n)}}.
\end{equation*}
\item[Step 2.b:] Update $\beta^{(n+1)}=\PP_{\mathcal{B}}\left(\beta^{(n)}-\delta_0\partial_{\beta}J^{(n)}\right)$.
\item[Step 2.c:] Compute $Y^{(n+1)},Z^{(n+1)},J^{(n+1)}$ and define $\dd J=J^{(n+1)}-J^{(n)}$.
\item[Step 2.d:] While $\dd J\geq 0$, i.e. the objective functional does not improve, then do the following
\begin{enumerate}[label=(2.d.\arabic*)]
\item Reduce the step length by 
\begin{equation*}
\delta_1 = \arg\min_{0<\delta< \delta_0}\phi(\delta):=J(\beta^{(n)}-\delta\partial_{\beta}J^{(n)})
\end{equation*}
where $\phi$ is an approximative quadratic function determined by the data: $\phi(0),\phi'(0)$ and $\phi(\delta_0)$ (see~\cite{GP2010i,WGS2014i}).
\item Update $\beta^{(n+1)}=\PP_{\mathcal{B}}(\beta^{(n)}-\delta_1\partial_{\beta}J^{(n)})$ and set $\delta_0=\delta_1$.
\item Compute $Y^{(n+1)},Z^{(n+1)},J^{(n+1)}$ and set $\dd J=J^{(n+1)}-J^{(n)}$.
\end{enumerate}
\item[Step 2.e:] Set $n=n+1$.
\end{description}
\end{mdframed}

\subsection{An approach from seasonal models}
The classical \textsc{ODE}--based SIRS model uses the simplest idea that the individuals are isolated and therefore, can only afford weak interactions with the others. The main aim in the model is to estimate of the infection rate $\beta$ and the recovery rate $\gamma$ using the given seasonal data. More importantly, we will use an appropriate estimate of $\gamma$ based on what is yielded from the seasonal models for the \textsc{PDE} model to focus only on the estimation of $\beta$, as indicated in previous discussions. 

\textbf{Hospitalized cases and meteorology}. Dengue fever and its severe variants remain one of the major health problems that threaten communities in Indonesia, including the city of Semarang. The Health Office of the City of Semarang has indicated in \cite{HOS2016i} that the number of hospitalized cases arose at times with increasing mobility and population density. Starting from January 1, 2010, the office has identified the endemicity of the diseases by collecting daily reports of hospitalized cases by mixed illnesses (Dengue Fever + Dengue Hemorrhagic Fever + Dengue Shock Syndrome) from $164$ villages (93\% of the city area). The surveyed villages had been seen to be endemic, covering approximately $1.5$ million residents or over $400$ thousands households. The monthly data summarized from daily reports starting from January 1, 2010 to April 30, 2015 is shown in the most upper part of Fig.~\ref{fig:meteorology}.

According to various studies, e.g. \cite{HE2007i,HC1997i,IMN2007i,SLP2007i}, seasonality in weather has been proven to influence the propagation of mosquito-borne diseases. It is essentially driven by the fact that weather influences the life cycle and breeding of mosquitoes carrying dengue viruses. Here, we gathered meteorology data for the Semarang station from the National Oceanic and Atmospheric Administration via Global Surface Summary of the Day \url{https://data.noaa.gov/dataset}, which associates with the website \url{https://www.wunderground.com}. As a correlation test, we have collected 6 parameters, all measured on the daily median basis: precipitation, dew point, pressure, temperature, visibility, and wind speed. Those data were collected only for the time window January 1, 2010 -- April 30, 2015 (1946 days) plus 30 days before January 1, 2010. The rationale of choosing up to 30-day backward shift is as follows. It needs several days after the rainfall to form standing pools, water containers, and other breeding sites for mosquitoes, until then some mosquitoes breed, grow, bite humans, and cause the diseases. Therefore, not at the day where the rainfall achieves its peak does the disease outbreak occur, but several days after (between 0 to average mosquito lifetime period $\approx$ 1 month). Moreover, most mosquito species likely feed at dawn and dusk, seemingly suggesting that the incidence negatively correlates with visibility. After summarizing all those daily data into monthly data, we found no correlation between the incidence data and wind speed. As for the other parameters, the highest correlations were found in different time shifts (see for the plots in Fig.~\ref{fig:meteorology}).

%fcorrDEWP =0.4658idx1 = 29fcorrSTP =0.1131idx2 = 7fcorrTEMP =-0.1123idx3 =1fcorrVISIB =  -0.2449idx4 =2fcorrPRCP =0.1858idx6 =28
%As for the other parameters, the highest correlations were found in different time shifts, as indicated in Fig.~\ref{fig:meteorology}.

\begin{figure}[htbp!]
\centering
\hspace*{-1cm}
\includegraphics[width=0.9\textwidth]{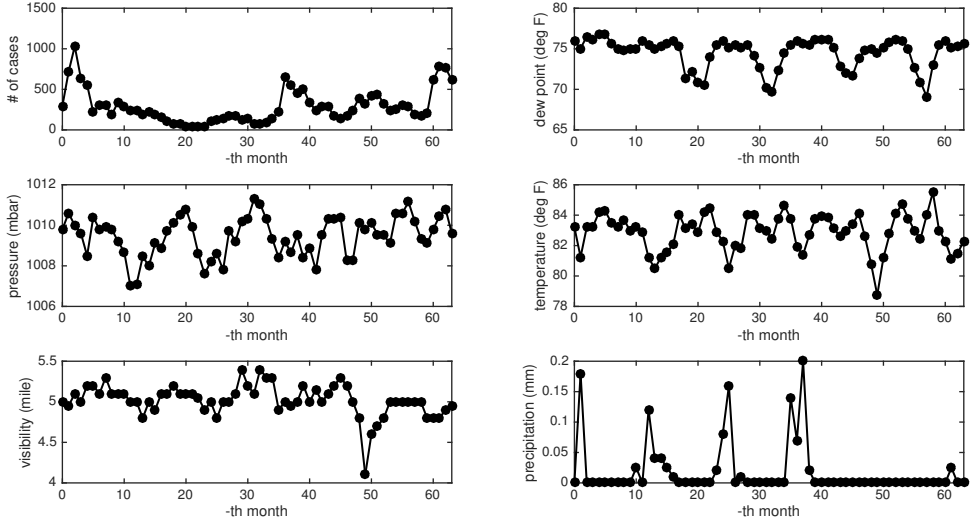}
\caption{The most correlated meteorological data with incidence case: dew point, pressure, temperature, visibility, and precipitation in 3, 24, 30, 29, and 3 days respectively before the incidence, with the correlations 0.4658, 0.1131, -0.1123, -0.2449, and 0.1858 respectively.}
\label{fig:meteorology}
\end{figure}

\textbf{Seasonal models}. We apply here the seasonal SIRS model that is based on the spatially homogeneous version of \eqref{eq:model}. As a first step we use only constant parameters including $\beta$ and $\gamma$. In the next step, we complement this model with its stochastic version that better helps pronounce the noise in the data. The current framework of adding state-dependent noise into the deterministic model was inspired by those in \cite{AB2000i,MA2006i}. Our stochastic model was based on the assumption that the difference in the states in $1$ individual before and after some time follow the \textsc{Gauss}ian distribution. We have distinguished all the events that appear based on the deterministic model, together with the corresponding probabilities for a small increment $\delta t$, in Table~\ref{tab:event}. One practical example to read the table is chosen from the third row: "the probability $\PM_3$ that the event $\lambda_3=(-1,1,0)'$ appears, i.e. the susceptible compartment $y_1$ decreases while at the same time the infective compartment $y_2$ increases by 1 individual after $\delta t$, is determined by the infection $\frac{\beta}{y}y_1y_2\delta t$".

\begin{table}[h!]
\centering
\begin{tabular}{c|c|c|c|c|c}
\hline
$j$ & Event & $\delta y_1$ & $\delta y_2$ & $\delta y_3$ & Probability $\PM_j$\\ \hline
1 & $\lambda_1$ & 1 & 0 & 0 & $\mu y\delta t$\\
2 & $\lambda_2$ & -1 & 0 & 0 & $\mu y_1\delta t$ \\
3 & $\lambda_3$ & -1 & 1 & 0 & $\frac{\beta}{y} y_1y_2\delta t$\\
4 & $\lambda_4$ & 0 & -1 & 0 & $\mu y_2 \delta t$\\
5 & $\lambda_5$ & 0 & -1 & 1 & $\gamma y_2 \delta t$\\
6 & $\lambda_6$ & 0 & 0 & -1 & $\mu y_3\delta t$\\ 
7 & $\lambda_7$ & 1 & 0 & -1 & $\kappa y_3\delta t$\\\hline
\end{tabular}
\caption{All the events appearing in the spatially homogeneous seasonal deterministic model~\eqref{eq:model}.}
\label{tab:event}
\end{table}
Owing to the notation for the deterministic model $\dd_t{Y}=F(Y)$, we may now generate the expectation $\mathbb{E}(\delta Y)=\sum_{j}\PM_j\lambda_j=F(Y)\delta t$ and the variance $\textnormal{Var}(\delta Y)=\mathbb{E}(\delta Y\delta Y')-(\mathbb{E}(\delta Y))^2\approx \mathbb{E}(\delta Y\delta Y')=\sum_{j}\PM_j\lambda_j\lambda_j\! '$ by considering that $\delta t$ is very small. Defining $q_j:=\PM_j\slash \delta t$ at the expense of extracting $\delta t$ and a vector of \textsc{Wiener} processes $\mathcal{W}=(W_1,W_2,W_3)$ such that $\frac{\delta \mathcal{W}}{\sqrt{\delta t}}=\frac{\mathcal{W}_{t+\delta t}-\mathcal{W}_t}{\sqrt{\delta t}}\sim \frac{N(0,\delta t)}{\sqrt{\delta t}}\sim N(0,1)$, we have 
\begin{equation*}
\delta Y = \mathbb{E}(\delta Y) + \sqrt{\textnormal{Var}(\delta Y)}\frac{\delta \mathcal{W}}{\sqrt{\delta t}}\stackrel{\delta t\rightarrow 0}{\longrightarrow}\dd Y = F(Y)\, \dd t + \sqrt{\VM} \,\dd \mathcal{W}.
\end{equation*}  
At a certain aim of preventing the noise making the solution trajectory jumps within high values from time to time, we use a damping factor $\rho\in\R$ such that 
\begin{equation}\label{eq:stocee}
\dd Y = F(Y)\, \dd t + \rho\sqrt{\VM} \,\dd \mathcal{W}.
\end{equation}
We shall note that $\rho$ has to be sensitive with the data. Now our model consists of both deterministic and noise parts. The variance matrix $\VM=\sum_{j}q_j\lambda_j\lambda_j\! '$ is given by
\begin{equation*}
\VM=\left(\begin{array}{ccc}
\mu y+\mu y_1+\frac{\beta}{y} y_1y_2 + \kappa y_3 & -\frac{\beta}{y} y_1y_2 & -\kappa y_3\\
-\frac{\beta}{y} y_1y_2 & \frac{\beta}{y} y_1y_2+\mu y_2 +\gamma y_2 & -\gamma y_2\\
-\kappa y_3 &-\gamma y_2 & \gamma y_2+\mu y_3+\kappa y_3
\end{array}
\right).
\end{equation*}
Computing the square root of $\VM$ as used in the model~\eqref{eq:stocee} requires diagonalization of $\VM$, which in our case results in very lengthy expressions of the variables. To mitigate this technical issue, we replace $\sqrt{\VM}$ with the ansatz matrix $G=G(Y)$ where $G_{ij}=\sqrt{q_j}\lambda_j^i$ such that it holds $\VM=GG'$. Taking back $\rho$ from its appearance, it turns out that the forward \textsc{Kolmogorov} equations (\textsc{Fokker--Plank} equations) for the model with $\sqrt{\VM}$ and that with $G$ are the same, leading to the same probability distributions for the two model solutions. In return, we now use
\begin{equation}\label{eq:stoceeuse}
\dd Y = F(Y)\,\dd t+ \rho G(Y)\,\dd \mathcal{W}.
\end{equation}
Using both deterministic and stochastic models, we specify the temporary cross immunity period $1\slash\kappa = 9$ months, the initial condition for susceptible $y_{1,0}=1.4$ million, the initial condition for infective $y_{2,0}=y_{2}^d(0)$ as the initial point in the seasonal data, and the human lifetime period $1\slash\mu = 65\times 12$ months. After several realizations, we came across a certain optimal value for the damping factor $\rho\approx 1.69\times 10^{-2}$. Beside $\gamma$ and $\beta$, we also aim at estimating the other remaining unknown: the initial condition of the recovered $y_{3,0}$. The realizations for both deterministic and stochastic models can be seen in Fig.~\ref{fig:realize}(a)--(b).
\begin{figure}[htbp!]
\subfigure[]{\includegraphics[width=0.35\textwidth]{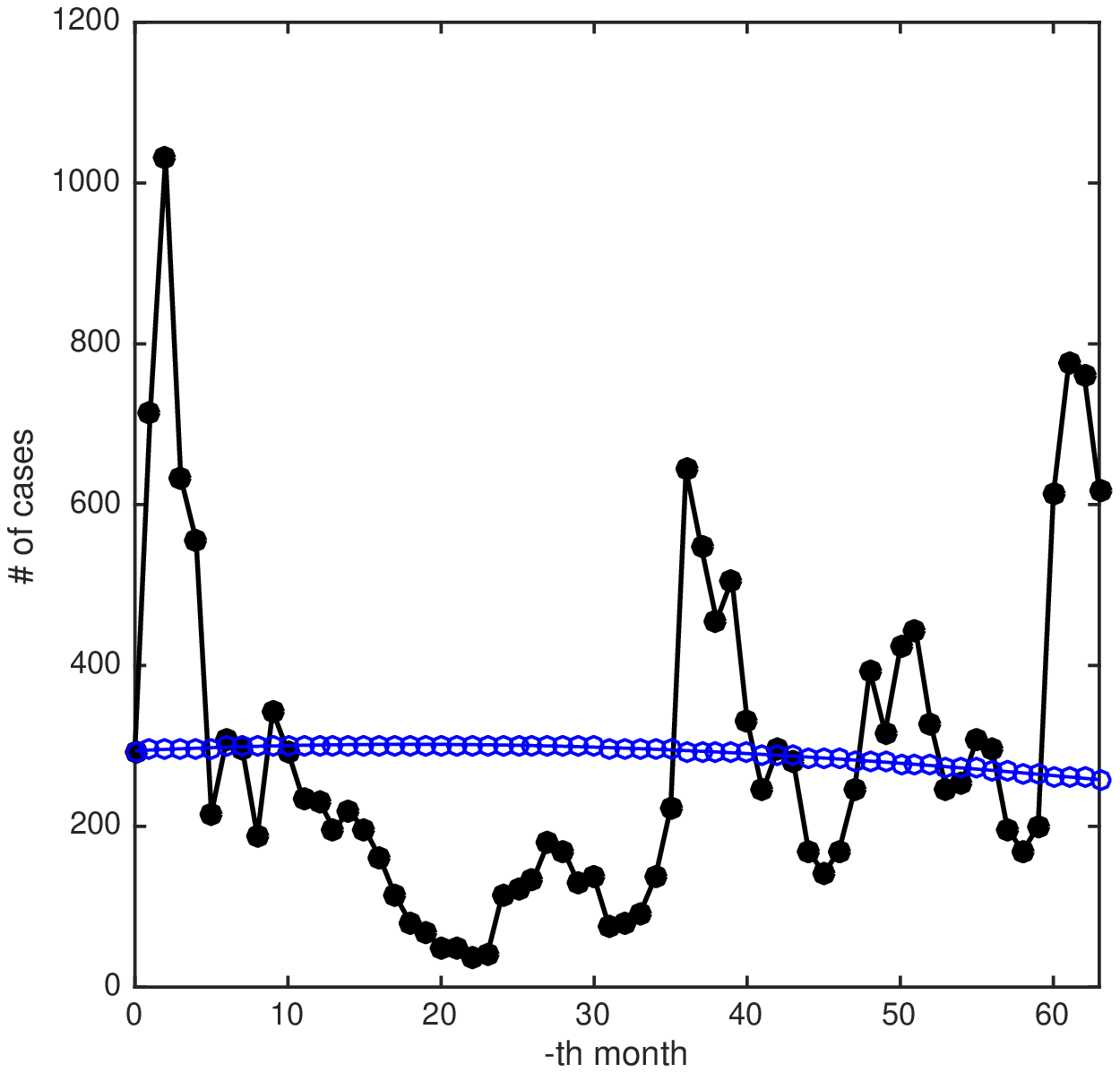}}
\hfill\hspace*{-1cm}
\subfigure[]{\includegraphics[width=0.35\textwidth]{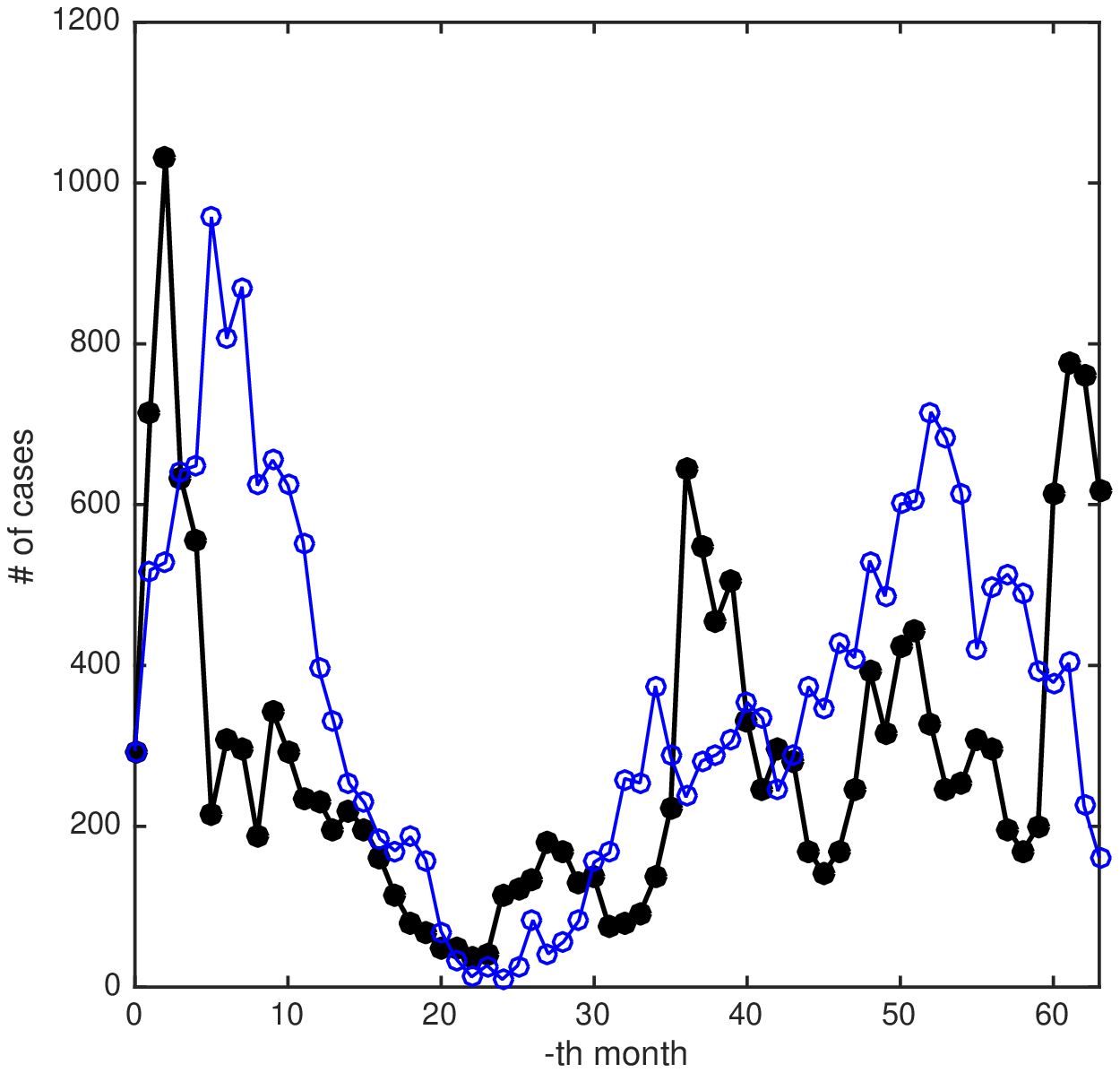}}
\hfill\hspace*{-1cm}
\subfigure[]{\includegraphics[width=0.35\textwidth]{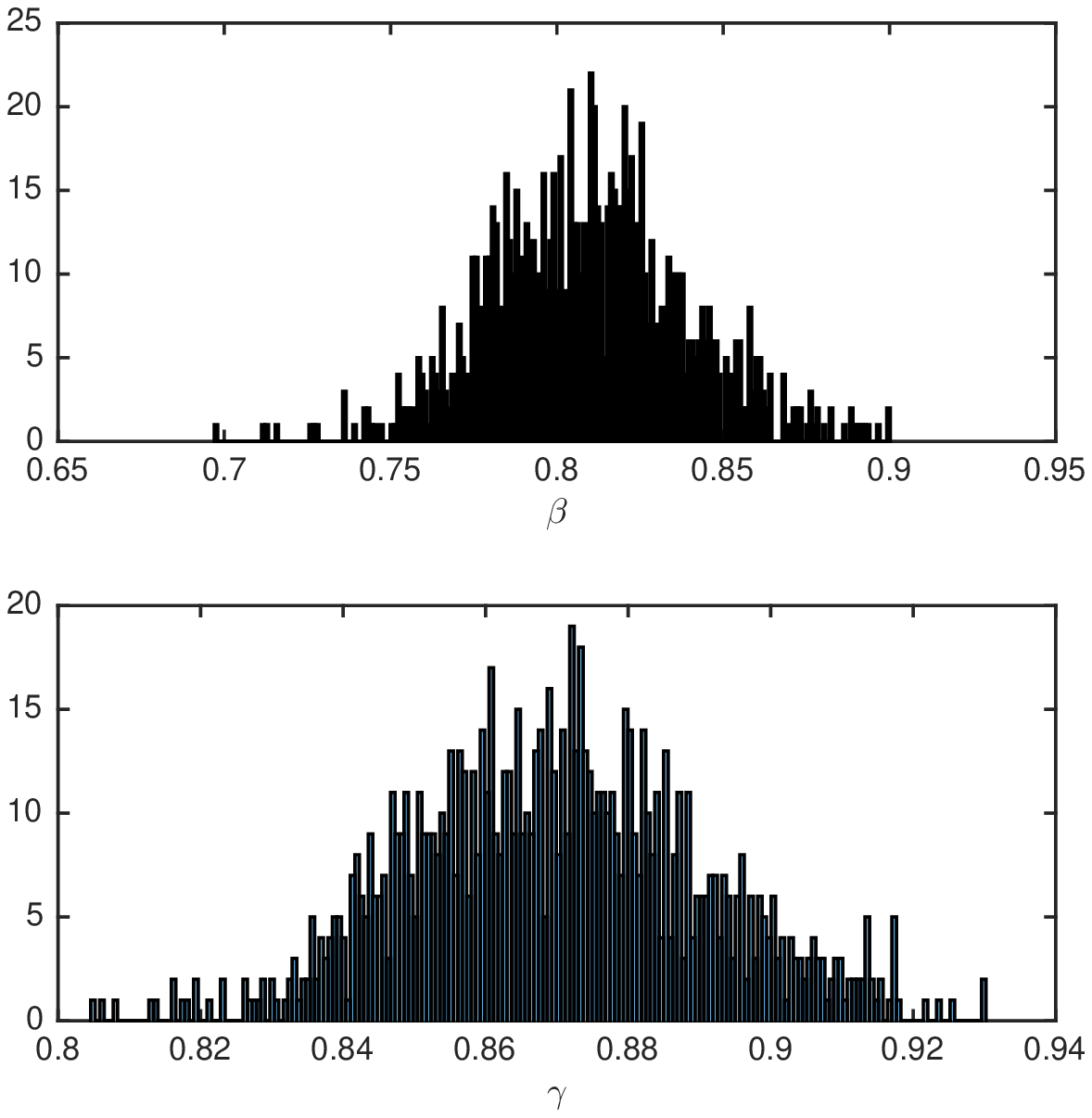}}
\hfill
\subfigure[]{\includegraphics[width=0.35\textwidth]{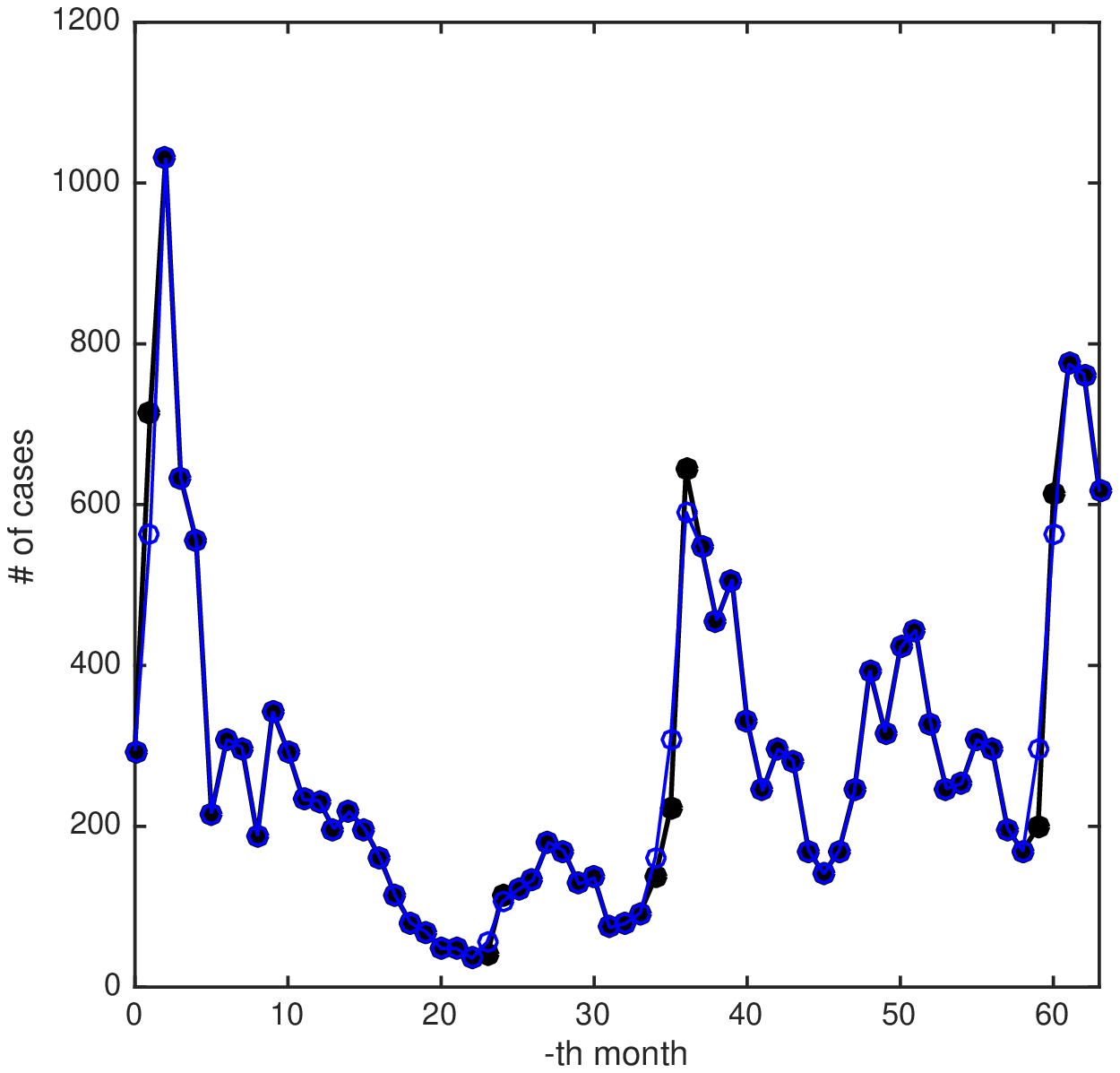}}
\hfill\hspace*{-1cm}
\subfigure[]{\includegraphics[width=0.35\textwidth]{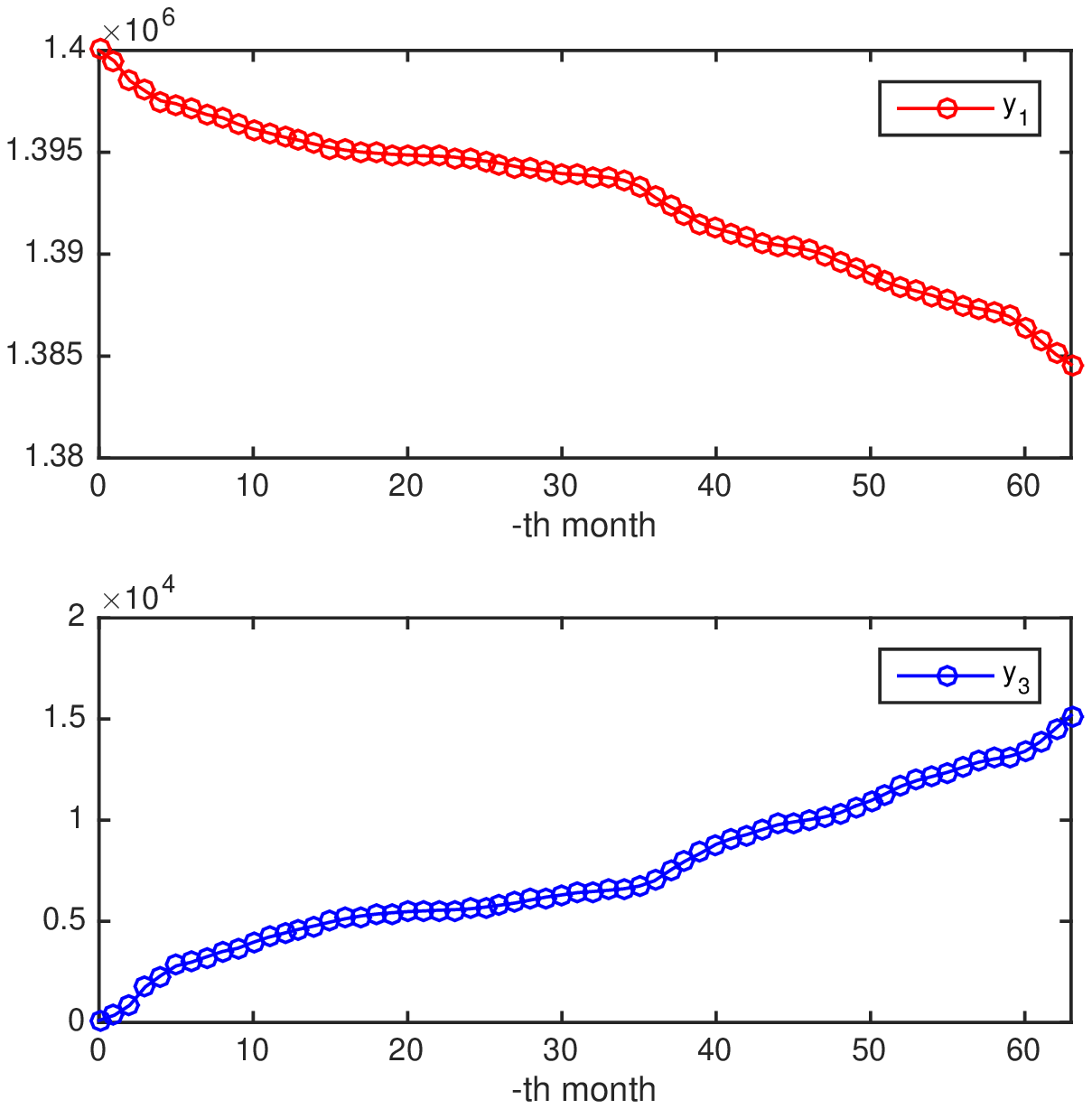}}
\hfill\hspace*{-1cm}
\subfigure[]{\includegraphics[width=0.35\textwidth]{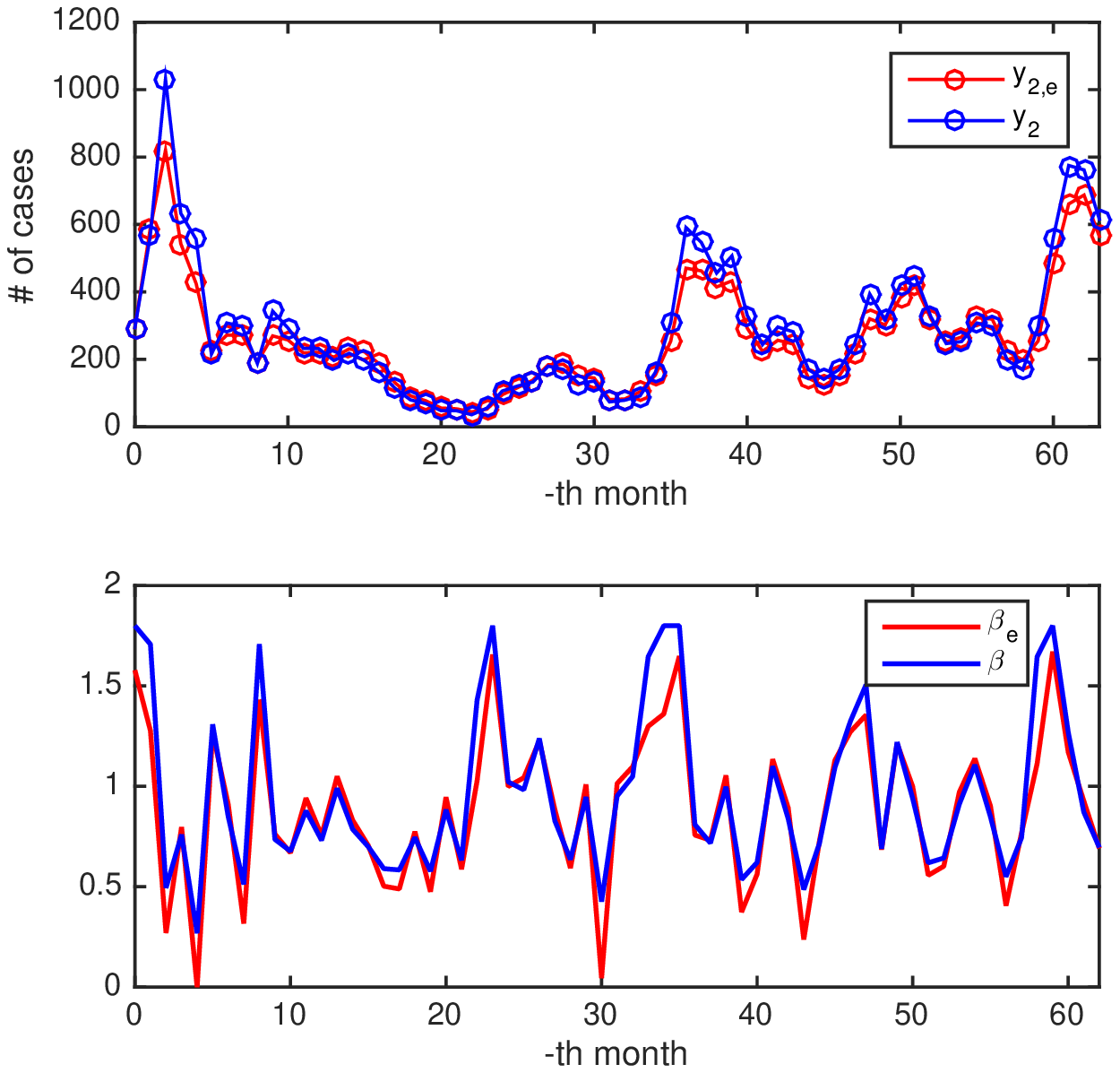}}
\caption{(a) Realization of the deterministic SIRS model (black curve: monthly incidence data, blue curve: model solution): estimated optimal parameters $\beta=0.8823$, $\gamma=0.8785$, $y_{3,0}\approx 227$ with the corresponding basic reproductive number $\frac{\beta}{\gamma+\mu}\lessapprox 1$. (b) One realization of the stochastic SIRS model~\eqref{eq:stoceeuse} with the damping factor $\rho\approx 1.69\times 10^{-2}$ and $y_{3,0}=227$. (c) Histogram of optimal $\beta$ and $\gamma$ from 1000 realizations of \eqref{eq:stoceeuse}. (d) Realization of the non-autonomous model by varying $\beta$ and giving $\gamma=0.879$, $y_{3,0}= 227$ in \eqref{eq:opteeode}. (e) Susceptible and recovered associated with infective population in (d). (f) Comparison of \eqref{eq:opteeode} and \eqref{eq:betaee}.}
\label{fig:realize}
\end{figure}

Conclusions are drawn from the two models. The drawback of the deterministic model is always that the data appears very noisy, whereas the solution regresses the data in between, leading to a higher metric between both terms. The advantages are that the model merely inherits computational simplicity and the basic reproductive number can be calculated. Using the given seasonal data, the basic reproductive number is found to be slightly less than $1$, leading the disease to die out in the long run. However, one may hardly rely on this threshold result since the tested data is contrarily defined on a relatively short period of time. Meanwhile, the stochastic model always requires a large number of solving processes in order to find the most reliable parameters. The solutions might track the data better than that of the deterministic model, but yet small gaps remain between the solution and the data due to the random walks. 

To mitigate the previous issues, we fix all the estimated parameters except $\beta$ and vary $\beta$ in time instead. The estimation is now based on the following fitting problem
\begin{align}\label{eq:opteeode}
	&\min_{\beta\in L^{\infty}(0,T)}\int_{(0,T)} \frac{1}{2}\bra{y_2-y_2^d}^2\,\dd t+\frac{\omega}{2}\int_{(0,T)} \beta^2\,\dd t\nonumber\\
	&\text{subject to }\beta_{\min}\stackrel{\aev}{\leq}\beta\stackrel{\aev}{\leq} \beta_{\max}\text{ and }\dd_t Y=F(Y,\beta(t)),\,Y(0)=Y_0. 
\end{align} 
Using the scheme in \eqref{eq:opteeode}, we see that the solution agrees with the data better than those from the previous two models, see Fig.~\ref{fig:realize}(d). Unlike the previous models, this approach opens a way to relate the optimal $\beta$ with several extrinsic factors, including meteorology. Apparently, the only variable that has correlations with the meteorology parameters is $y_2^d$ (consider Fig.~\ref{fig:meteorology}). The following ansatz
\begin{equation}\label{eq:betaee}
\beta_{e}=\bra{\frac{\dd_t y_2^d}{y_2^d}+\gamma+\mu}\frac{y}{y_1}
\end{equation}
has a direct correlation with $y_2^d$ for a sufficiently large $y_1$, i.e. $y\slash y_1=\text{ constant }\approx 1$. Approximating this constant and using $\beta_{e}$, we obtain a solution trajectory $y_{2,e}$ that well agrees with $y_2$ from \eqref{eq:opteeode}, see Fig.~\ref{fig:realize}(f). In such a case, finding correlations between $\beta$ and the meteorology parameters is doable upon performing a prediction.

\subsection{\textsc{PDE} model}

\begin{figure}[hbbp!]
\centering
\includegraphics[scale=0.23]{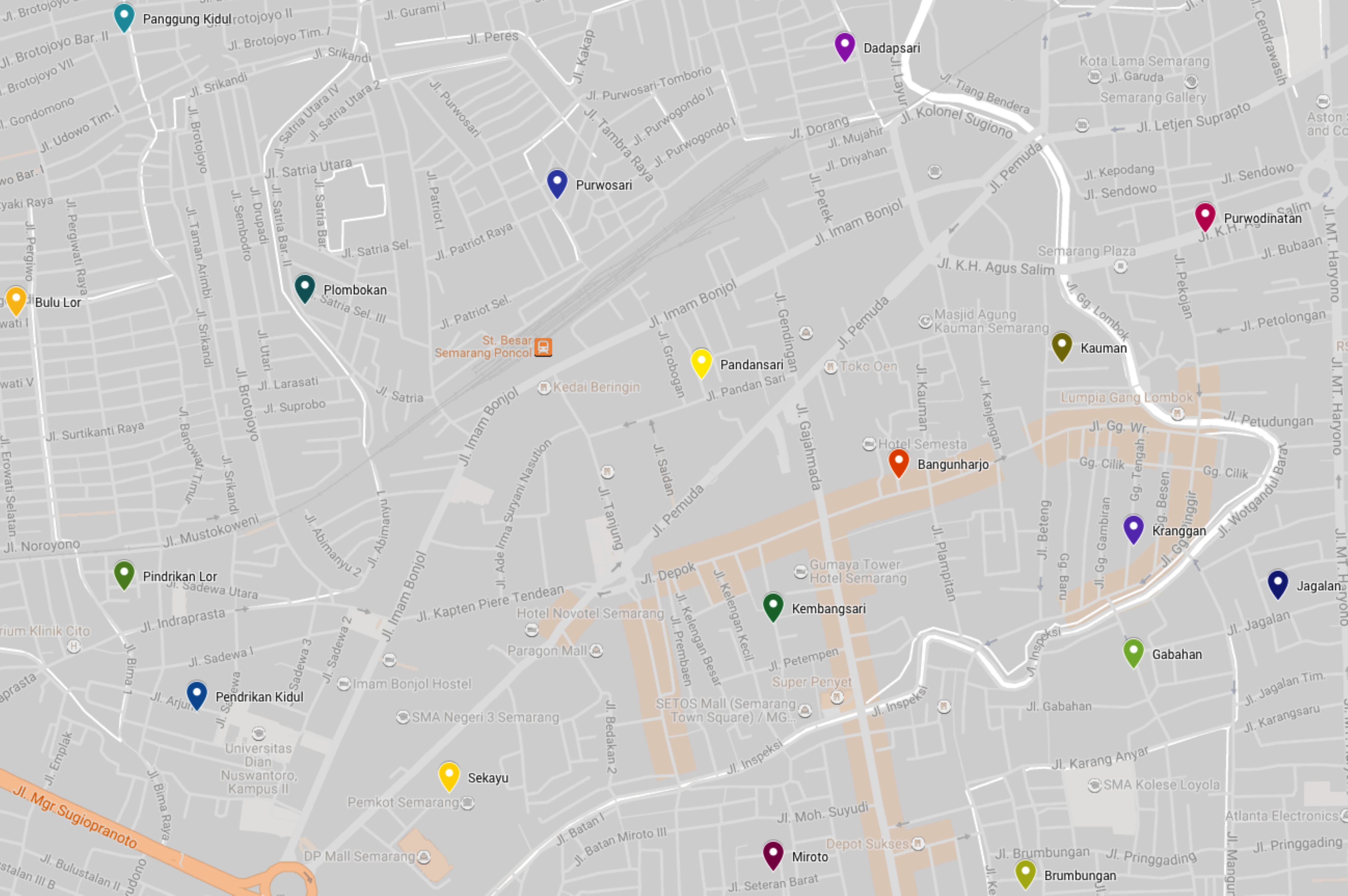}
\caption{18 villages in the city of Semarang for the test problem. The captured domain extends from $0$ km to $3.15$ km in the horizontal axis and from $0$ km to $1.69$ km in the vertical axis.}
\label{fig:dataPDE}
\end{figure}

\textbf{Seasonal-spatial data}. As for testing our \textsc{PDE} model \eqref{eq:model}, we are also supplemented with seasonal-spatial data. The data covers the yearly hospitalized cases (Dengue Fever + Dengue Hemorrhagic Fever + Dengue Shock Syndrome) from 2009 to 2014 in over 160 villages. For the test problem, we use the data from 18 villages as indicated in Fig.~\ref{fig:dataPDE}. 

In the numerical computation, we use year as the time unit. A graphical approximation using an image processing technique allows us to perform an empirical measurement for the spatial domain as in Fig.~\ref{fig:dataPDE} and locate the points indicating the villages. We then rescale the actual domain $(0,3.15)\times (0,1.69)$ (in km) into a unit domain
\begin{equation*}
\Omega = (0,1)\times(0,1)
\end{equation*}
by the divisions $(0,3.15)\slash 3.15$ in the horizontal axis and $(0,1.69)\slash 1.69$ in the vertical axis. Therefore, the length now becomes dimensionless. We assume that susceptible and recovered people are of the same mobility, i.e $d_1=d_3=d$, while the infective people are a bit static, i.e. $d_2=\epsilon d$ for some $\epsilon\in (0,1)$. In a brief presentation, all the parameters are specified as shown in Table~\ref{tab:paramee}. The value of $\gamma$ is derived from its optimal value obtained from the seasonal models.
\begin{table}[htbp!]
\resizebox{\columnwidth}{!}{%
\begin{tabular}{r|c|c|c|c|c|c|c|c|c|c|c}
\hline
Par. & $\Omega$ & $1\slash\mu$ & $1\slash\gamma$ & $1\slash\kappa$ & $y_{1,0}$ & $y_{2,0}$ & $y_{3,0}$ & $1\slash d$ & $\epsilon$ & $\beta_{\min}$ & $\beta_{\max}$\\ %\hline
Unit & - & year & year & year & ind. & ind. & ind. & year & - & year$^{-1}$ & year$^{-1}$\\ %\hline
Value & $(0,1)\times(0,1)$ & $65$ & $\frac{0.879}{12}$ & $9\slash 12$ & $200$ & $y^d_2(0)$ & $3$ & $10^3\slash 5$ & $0.2$ & $0$ & $4$\\ \hline
\end{tabular}
}
\caption{Parameters used in the numerical computation.}
\label{tab:paramee}
\end{table}

%. A graphical approximation using an image processing technique allows us to specify the length unit for the spatial domain in Fig.~\ref{fig:dataPDE}, given under uniformity assumption $\delta x_1= 0.45$ km and $\delta x_2=0.3375$ km. In the numerical computation, we a rescaled domain comprising $3.15\slash 0.45\approx 7$ partitions in the horizontal axis and $1.69\slash 0.3375\approx 5$ partitions in the vertical axis.

\begin{figure}[htbp!]
\hspace*{-1.5cm}
\centering
\includegraphics[width=1.2\textwidth]{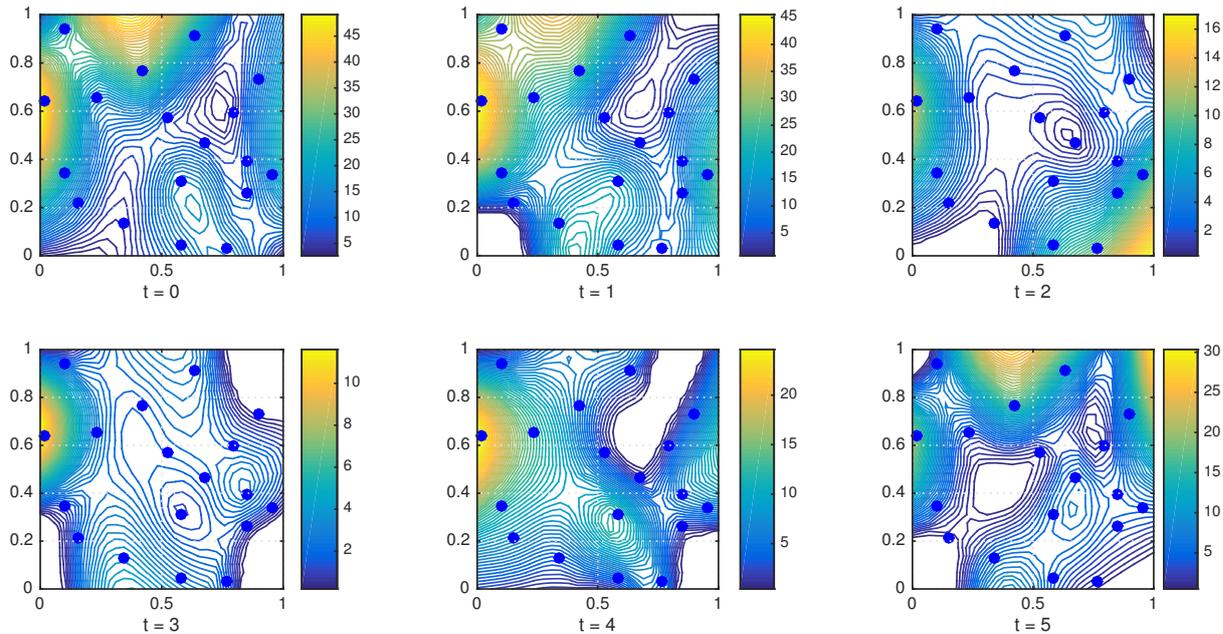}
\caption{The epidemic data after rescaling the domain in Fig.~\ref{fig:dataPDE} to $\Omega$, interpolation and extrapolation, and saturating the negative values into zero.}
\label{fig:DATAee}
\end{figure} 

\textbf{Interpolation and extrapolation}. Since the spatial locations of the villages do not fit with those characterized by the discretization of the system~\eqref{eq:model}, one practical aid is to perform interpolation and extrapolation. Interpolation, on the one hand, helps provide the "projections" of the data points into the specified equidistant space locations or grid-points that are located in the convex hull of the data points. On the other hand, extrapolation extends the projections to all spatial locations in the domain that are beyond the convex hull. Typical inter- and extrapolation usually introduce a flexible surface defined on the whole domain that stands in between the data points. Those data points can be connected to the points in the surface at the same spatial locations via some bands. As the bands are contracted, the surface starts to deform with which the connected points become closer to each other. At this stage, the so-called \emph{relative stiffness} of the surface plays a role in maintaining the smoothness of the surface against the deformation towards a noisy structure of the data. The \textsc{Matlab} toolbox \textsc{Gridfit} \cite{Err2016i} provides this mechanism as a prominent approach in the extrapolation as well as controllable trade-off between the relative stiffness and contraction of the bands connecting the data points and the surface. In the numerical computations, all interpolation and extrapolation tasks were done using this toolbox. As a benchmarking result, the inter- and extrapolated data points using this toolbox are shown in Fig.~\ref{fig:DATAee}. 

The tracking result as well as the distribution of optimal infection rate are presented in Fig.~\ref{fig:resPDE}. We put the corresponding comments for this result in the conclusion. 

\begin{figure}[htbp!]
\hfill\hspace*{-1.5cm}
\subfigure[]{\includegraphics[width=1.2\textwidth]{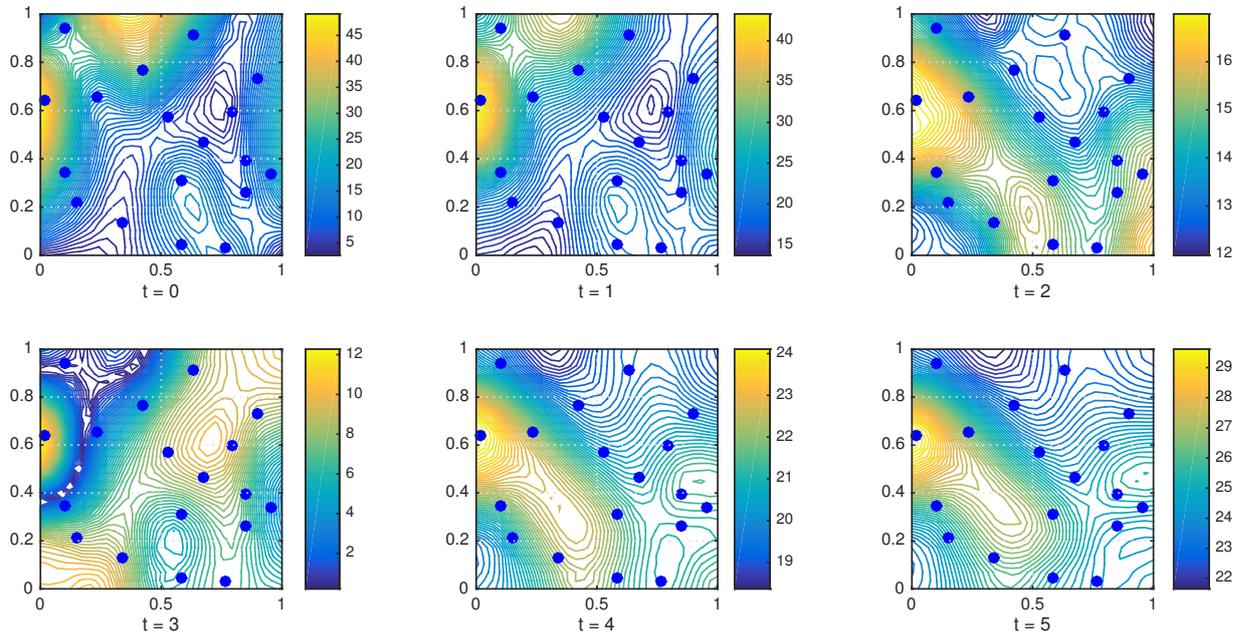}}
\hfill\hspace*{-1.5cm}
\subfigure[]{\includegraphics[width=1.2\textwidth]{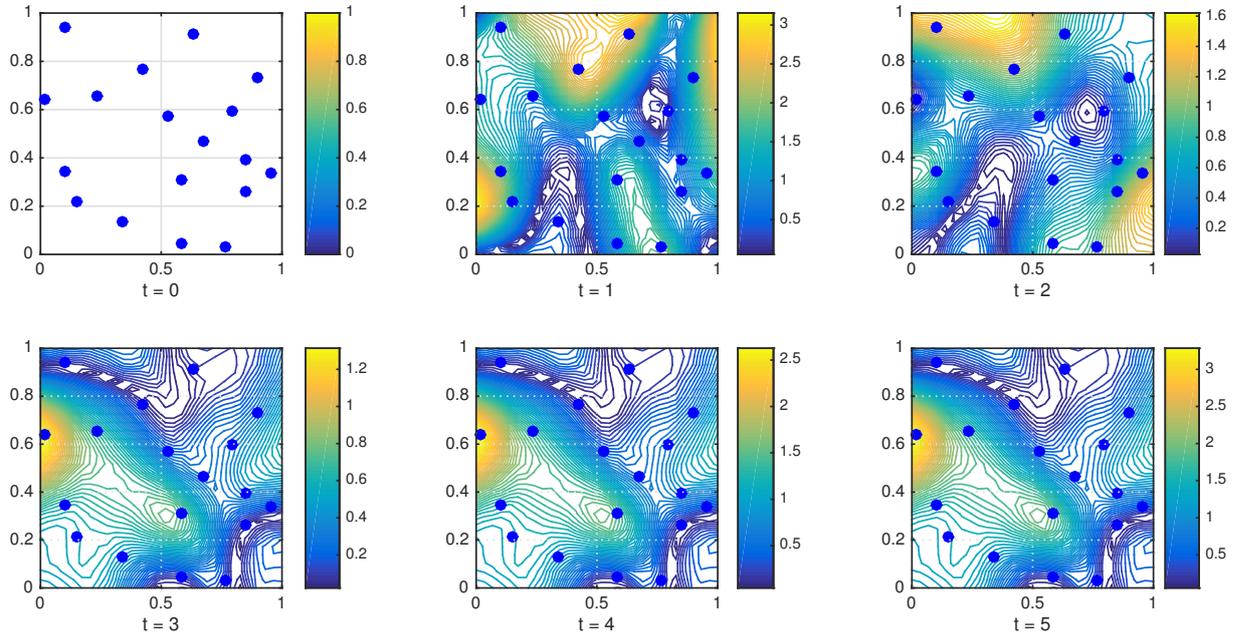}}
\hfill
\caption{(a) The optimal $y_2$ and (b) optimal $\beta$ from time to time with $\omega=10^{-3}$. Note that $\beta$ at $t=0$ is no use since the initial condition was already specified from the given data.}
\label{fig:resPDE}
\end{figure}

%\begin{figure}[htbp!]
%\hfill\hspace*{-1.5cm}
%\subfigure[]{\includegraphics[width=1.2\textwidth]{PDEdata.eps}}
%\hfill\hspace*{-1.5cm}
%\subfigure[]{\includegraphics[width=1.2\textwidth]{PDEdata.eps}}
%\hfill
%\caption{(a) The optimal $y_2$ and (b) optimal $\beta$ from time to time with $\omega=10^{-12}$.}
%\label{fig:resPDEb}
%\end{figure}

%The infected peak spreads initially, then in a later time diminishes and it forms an infected wave that propagates to the boundary to finally disappears. It is important to recall that the diffusion constant is contained in the scaling length (d/Na)1/2 so the spatial behavior should be generic. If ? were greater than s0 no epidemic develop and the peak diminishes until it vanishes. But if ? < s0 the infected density increases and an epidemic develops.

\section{Concluding remarks}
\label{sec:conclusion}
In this work, we use two SIRS models to track two datasets from the city of Semarang, Indonesia. The two datasets are different in dimensions, so are the two models: seasonal model and seasonal-spatial model. The seasonal model assumes that the people are distant, since then they interact very weakly. Also, it does not consider that the disease can spread spatially from a locality to other localities and that it takes time to spread across localities. Three variants of the model were proposed based on the definition of the infection rate $\beta$: deterministic with constant $\beta$, stochastic with constant $\beta$, and deterministic with time-varying $\beta$. The computation results show that the first two models require less computational efforts but suffer from tremendous gaps between the data and the actual solutions. The model with time-varying $\beta$ may be more time consuming, but the result outperforms the other models regarding the gaps.

We have improved the SIRS model in order to take into account the spatial behavior of all the sub-populations by adding diffusion terms. We use this model to study the spatial dynamics of the dengue epidemics in the city of Semarang and assume that the people are closely situated, which is a rather stringent condition but however can give us some illustration of the spatial and temporal behavior of a disease epidemics. By optimizing time-space-varying $\beta$, we were able to track the actual solution to the second dataset. Due to some reasons, we can no longer use a value of the regularization parameter $\omega$ that is less than $10^{-3}$, since then it would lead the optimal infection rate to reach seemengly unrealistic levels of more than $3$. Moreover, at this stage we see that performing a prediction is not immediate. Finding a relationship between the optimal $\beta$ and some locally unique extrinsic factors, as also discussed in the seasonal models, is addressed as a subject of future research.

Other possible improvements for the model can include the migrations of the people. This will lead to a no-boundary-condition problem, which of course increases the complexity of the problem, and therefore, increases the computational efforts. In another direction of interest, one can also embrace the aforementioned models with cross-diffusions.

\section*{Acknowledgements}
This research is supported by DAAD from the project MMDF "Mathematical Models for Dengue Fever" project-ID 57128360 (DAAD PPP--Portugal). We thank Sutimin from the Department of Mathematics, Diponegoro University, Indonesia for providing us with the incidence data during his visit to the University of Koblenz on November 2015. The data has been used under the authority of the Health Office of the City of Semarang.

% Non-BibTeX users please use

\end{document}